\documentclass[reqno, 12pt]{amsart}
\makeatletter
\let\origsection=\section \def\section{\@ifstar{\origsection*}{\mysection}} 
\def\mysection{\@startsection{section}{1}\z@{.7\linespacing\@plus\linespacing}{.5\linespacing}{\normalfont\scshape\centering\S}}
\makeatother        
\usepackage{amsmath,amssymb,amsthm}    
\usepackage{mathabx}\changenotsign    
\usepackage{mathrsfs} 
\usepackage{dsfont}

\usepackage{xcolor}  	
\usepackage[backref]{hyperref}
\hypersetup{
	colorlinks,
    linkcolor={red!60!black},
    citecolor={green!60!black},
    urlcolor={blue!60!black},
}

\usepackage{bookmark}

\usepackage[abbrev,msc-links,backrefs]{amsrefs} 
\usepackage{doi}

\renewcommand{\PrintDOI}[1]{\doi{#1}}

\usepackage[babel]{microtype}
\usepackage[T1]{fontenc}
\usepackage{lmodern}

\usepackage[english]{babel}

\usepackage{geometry}
\usepackage{enumitem}

\let\polishlcross=\l
\def\l{\ifmmode\ell\else\polishlcross\fi}

\def\tand{\ \text{and}\ }
\def\qand{\quad\text{and}\quad}
\def\qqand{\qquad\text{and}\qquad}

\let\emptyset=\varnothing
\let\setminus=\smallsetminus

\makeatletter
\def\moverlay{\mathpalette\mov@rlay}
\def\mov@rlay#1#2{\leavevmode\vtop{   \baselineskip\z@skip \lineskiplimit-\maxdimen
   \ialign{\hfil$\m@th#1##$\hfil\cr#2\crcr}}}
\newcommand{\charfusion}[3][\mathord]{
    #1{\ifx#1\mathop\vphantom{#2}\fi
        \mathpalette\mov@rlay{#2\cr#3}
      }
    \ifx#1\mathop\expandafter\displaylimits\fi}
\makeatother

\newcommand{\dcup}{\charfusion[\mathbin]{\cup}{\cdot}}

\makeatletter
\newcommand{\subalign}[1]{  \vcenter{    \Let@ \restore@math@cr \default@tag
    \baselineskip\fontdimen10 \scriptfont\tw@
    \advance\baselineskip\fontdimen12 \scriptfont\tw@
    \lineskip\thr@@\fontdimen8 \scriptfont\thr@@
    \lineskiplimit\lineskip
    \ialign{\hfil$\m@th\scriptstyle##$&$\m@th\scriptstyle{}##$\crcr
      #1\crcr
    }  }
}
\makeatother  

\newtheorem{theorem}{Theorem}[section]

\newtheorem{lemma}[theorem]{Lemma}

\newtheorem{proposition}[theorem]{Proposition}

\newtheorem{claim}{Claim}

\newtheoremstyle{definition}  {4pt}  {4pt}  {\sl}  {}  {\bfseries}  {.}  {.5em}          {}
\theoremstyle{definition}
\newtheorem{definition}[theorem]{Definition}

\newtheorem{conjecture}[theorem]{Conjecture}

\newtheorem{case}{Case}\usepackage{chngcntr}
\counterwithin*{case}{subsection}
\newtheorem{subcase}{Case}[case]

\theoremstyle{remark}
\newtheorem{remark}[theorem]{Remark}

\newtheoremstyle{introthms}  {3pt}  {3pt}  {\itshape}  {}  {\bfseries}  {.}  {.5em}          {\thmnote{#3}}\theoremstyle{introthms}

\def\paragraph#1{  \smallskip  \noindent\textbf{#1.}\enspace}

\def\eqqed{\pushQED{\qed}\qedhere}

\let\eps=\varepsilon
\let\theta=\vartheta
\let\rho=\varrho
\let\phi=\varphi

\def\iti#1{{\rm ({\it #1\,})}}

\def\NN{\mathds N}

\def\RR{\mathds R}

\def\cA{{\mathcal A}}

\def\cL{{\mathcal L}}
\def\cF{{\mathcal F}}
\def\cB{{\mathcal B}}
\def\cC{{\mathcal C}}

\def\cS{{\mathcal S}}

\def\sS{{\mathscr{S}}}

\def\hW{\hat W}
\def\hZ{\hat Z}

\newcommand{\PP}[1]{{\mathds P}\left(#1\right)}
\newcommand{\EE}[1]{{\mathds E}\left[#1\right]}
\newcommand{\var}[1]{\text{Var}\left[#1\right]}

\def\bp{\boldsymbol{p}}
\def\bq{\boldsymbol{q}}

\DeclareMathOperator{\ex}{ex}
\DeclareMathOperator{\rank}{rank}
\def\tand{\ \text{and}\ }
\def\bH{\boldsymbol{H}}

\def\bdel{\hat b}
\def\ndel{\hat n}

\begin{document}
\title{Extremal results for random discrete structures}

\author{Mathias Schacht}
\address{Fachbereich Mathematik, Universit\"at Hamburg, 
  Hamburg, Germany}
\email{schacht@math.uni-hamburg.de}
\thanks{Author was supported through the \emph{Heisenberg-Programme} of the DFG\@.}



\date{October 16, 2009 (submitted), November 25, 2014 (revised), April 7, 2016 (accepted)}

\begin{abstract}
We study thresholds for extremal properties of random discrete structures.
We determine the threshold for Szemer\'edi's theorem 
on arithmetic progressions in random subsets of the integers and its multidimensional extensions 
and we determine the threshold for Tur\'an-type problems for random graphs and hypergraphs.
In particular, we verify a conjecture of Kohayakawa, \L uczak, and R\"odl for Tur\'an-type problems 
in random graphs.
Similar results were obtained independently by Conlon and Gowers.
\end{abstract}

\maketitle

\section{Introduction}
Extremal problems are widely studied in discrete mathematics. Given a finite set $\Gamma$ 
and a family $\cF$ of subsets of $\Gamma$ an \emph{extremal result} asserts that 
any sufficiently large (or dense) subset $G\subseteq \Gamma$ must contain an element from $\cF$.
Often all elements of $\cF$ have the same size, i.e., $\cF\subseteq \binom{\Gamma}{k}$ for some integer $k$,
where $\binom{\Gamma}{k}$ denotes the family of all $k$-element subsets of $\Gamma$.

For example, if  $\Gamma_n=[n]=\{1,\dotsc,n\}$ and $\cF_n$ consists of  
all $k$-element subsets of~$[n]$ which form an arithmetic progression, then
Szemer\'edi's celebrated theorem~\cite{Sz75} asserts 
that every subset $Y\subseteq [n]$ with $|Y|=\Omega(n)$ 
contains an arithmetic progression of length~$k$.

A well known result from  graph theory, which fits this framework, is Tur\'an's theorem~\cite{Tu41} 
and its generalization due to Erd\H os and Stone~\cite{ErSt46} (see also~\cite{ErSi66}). 
Here $\Gamma_n=E(K_n)$ is the edge set of the complete graph with $n$ vertices and 
$\cF_n$ consists of the edge sets of copies of some fixed graph $F$ (say with $k$ edges) in $K_n$.
Here the Erd\H os-Stone theorem implies that every subgraph $H\subseteq K_n$ 
which contains at least 
$$
\left(1+\frac{1}{\chi(F)-1}-o(1)\right)\binom{n}{2}
$$ 
edges must contain a copy of $F$, where 
$\chi(F)$ denotes the chromatic number of $F$ (see, e.g.,~\cites{Bo78,Bo98,BoMu08,Di10}). The connection with 
the chromatic number was explicitly stated in the work of Erd\H os and Simonovits~\cite{ErSi66}.

We are interested in ``random versions'' of such extremal results. 
We study the
 \emph{binomial model} of random substructures.
 For a finite set  $\Gamma_n$ and a probability $p\in[0,1]$ we denote by 
$\Gamma_{n,p}$ the random subset where every $x\in\Gamma_n$ is included 
in $\Gamma_{n,p}$ independently with probability $p$. In other words, 
$\Gamma_{n,p}$ is the finite probability space on the power set of $\Gamma_n$
in which every elementary event $\{G\}$ for $G\subseteq \Gamma_n$ occurs with probability
$$
\PP{G=\Gamma_{n,p}}=p^{|G|}(1-p)^{|\Gamma_n|-|G|}\,.
$$
For example, if $\Gamma_n$ is the edge
set of the complete graph on $n$ vertices, then $\Gamma_{n,p}$ denotes the usual 
binomial random graph $G(n,p)$ (see, e.g.,~\cites{Bo01,JLR00}).

The deterministic extremal results mentioned earlier can be viewed as statements which hold with 
probability $1$ for $p=1$ and it is natural to investigate the asymptotic of the smallest probabilities
for which those results hold. In the context of Szemer\'edi's theorem
for every $k\geq 3$ and $\eps>0$ we are interested in the smallest sequence $\bp=(p_n)_{n\in\NN}$ 
of probabilities
such that the binomial random subset $[n]_{p_n}$ has  asymptotically almost surely
(a.a.s., i.e.\ with probability tending to $1$ as $n\to\infty$) the following property:
Every subset $Y\subseteq [n]_{p_n}$ with $|Y|\geq \eps |[n]_{p_n}|$ contains an arithmetic progression 
of length~$k$. Similarly, in the context of the Erd\H os-Stone theorem, 
for every graph $F$ and $\eps>0$ we are interested in the asymptotic of the smallest 
sequence $\bp=(p_n)_{n\in\NN}$  such that the random graph $G(n,p_n)$ a.a.s.\ satisfies:
every $H\subseteq G(n,p)$ with 
$$
e(H)\geq \left(1-\frac{1}{\chi(F)-1}+\eps\right)e(G(n,p_n))\,,
$$ 
contains a copy of~$F$.

We determine the asymptotic growth of the smallest such sequence $\bp$ of probabilities 
for those and some related extremal properties including multidimensional versions of 
Szemer\'edi's theorem (Theorem~\ref{thm:FK}), 
solutions of density regular systems of equations (Theorem~\ref{thm:dr}), 
an extremal version for solutions of the  Schur equation (Theorem~\ref{thm:Schur}), and extremal problems for hypergraphs (Theorem~\ref{thm:Tur}). 
In other words, 
we determine the \emph{threshold} for those properties. Similar results were obtained by Conlon and Gowers~\cite{CG}.

The new results will 
follow from a general result (see Theorem~\ref{thm:main}), 
which allows us to transfer certain extremal results from the classical deterministic setting 
to the probabilistic setting. In Section~\ref{sec:main_results_pf} we deduce the results 
stated in the next section from Theorem~\ref{thm:main}.

\section{New results}
\label{sec:main_results}

\subsection{Szemer\'edi's theorem and its multidimensional extension}
\label{sec:szem}
We study extremal properties of random subsets of the first $n$ positive integers.
One of the best known extremal-type results for the integers is Szemer\'edi's theorem.
In 1975 Szemer\'edi solved a longstanding conjecture of Erd\H os and Tur\'an~\cite{ErTu36} by showing 
that every subset of the integers of upper positive density contains 
an arithmetic progression of any finite length.
For a set $X\subseteq [n]$ we write
\begin{equation}\label{eq:notSz}
X\rightarrow_\eps [k]
\end{equation}
for the statement that every subsets $Y\subseteq X$ with $|Y|\geq \eps |X|$ contains an arithmetic progression 
of length~$k$. With this notation at hand, we can state (the finite version of) Szemer\'edi's theorem as follows:
for every integer $k\geq 3$ and $\eps>0$ there exists $n_0$ such that for every $n\geq n_0$ 
we have $[n]\rightarrow_\eps [k]$.

For fixed $k\geq 3$ and $\eps>0$ we are interested in the asymptotic behavior of the 
threshold sequence of probabilities $\bp=(p_n)$ such that 
there exist constants $0<c<C$ for which
\begin{equation}\label{eq:SzThr}
\lim_{n\to\infty}\PP{[n]_{q_n}\rightarrow_\eps [k]}=
\begin{cases} 
0, & \text{if}\ q_n\leq cp_n\ \text{for all}\ n\in\NN,\\
1, & \text{if}\ q_n\geq Cp_n\ \text{for all}\ n\in\NN.
\end{cases}
\end{equation}
\begin{remark}
We note that the family $\{X\subseteq [n]\colon X\rightarrow_\eps [k]\}$  is not 
closed under supersets. In other words, the property is ``$X\rightarrow_\eps [k]$''
is not a monotone property. However, similar arguments as presented in~\cite{JLR00}*{Proposition~8.6}
show that the property ``$X\rightarrow_\eps [k]$'' and the other properties considered in this 
section have a threshold as displayed in~\eqref{eq:SzThr}.
\end{remark}

It is easy to see that if the expected number of arithmetic progressions 
of length~$k$ in~$[n]_{q_n}$ is asymptotically smaller than the expected number of elements in 
$[n]_{q_n}$, then there exists a subset of size $(1-o(1))|[n]_{q_n}|$, which contains no arithmetic progressions 
of length $k$ at all. In other word, if 
\begin{equation}\label{eq:thrSz}
q_n^kn^2\ll q_nn
\qquad\Longleftrightarrow\qquad
q_n\ll n^{-1/(k-1)}
\end{equation}
then $\PP{[n]_{q_n}\rightarrow_\eps [k]}\to 0$ for every $\eps<1$. Consequently, $n^{-1/(k-1)}$
is a lower bound on the threshold for Szemer\'edi's theorem for arithmetic progressions of length $k$.
For $k=3$ Kohayakawa, \L uczak, and 
R\"odl~\cite{KLR96} established a matching upper bound. Our first result generalizes this 
for arbitrary $k\geq 3$.

\begin{theorem}
\label{thm:Sz}
For every integer $k\geq 3$ and every $\eps\in(0,1)$ there exist constants $C>c>0$ such that
for any sequence of probabilities $\bq=(q_n)_{n\in\NN}$ we have 
$$
\lim_{n\to\infty}\PP{[n]_{q_n}\rightarrow_\eps [k]}=
\begin{cases} 
0, & \text{if}\ q_n\leq cn^{-1/(k-1)}\ \text{for all}\ n\in\NN,\\
1, & \text{if}\ q_n\geq Cn^{-1/(k-1)}\ \text{for all}\ n\in\NN.
\end{cases}
$$
\end{theorem}
We remark that the 0-statement in Theorem~\ref{thm:Sz} 
(and, similarly, the 0-statements of the other results of this section) 
follows from standard probabilistic arguments. The 1-statement of 
Theorem~\ref{thm:Sz} follows from our main  result, Theorem~\ref{thm:main}.

A multidimensional version of Szem\'eredi's theorem was 
obtained by Furstenberg and Katznelson~\cite{FuKa78}. Those authors showed that 
for every integer $\l$, every finite subset $F\subset \NN^\l$ and every $\eps>0$ there exists some integer 
$n_0$ such that for $n\geq n_0$ every $Y\subseteq [n]^\l$ with $|Y|\geq \eps n^\l$
contains a homothetic copy of $F$, i.e., there exist some $y_0\in\NN^\l$
and $\lambda> 0$ such that $y_0+\lambda F=\{y_0+\lambda f\colon f\in F\}\subseteq Y$.
Clearly, the case $\l=1$ and $F=[k]$ resembles Szemer\'edi's theorem. 
Generalizing the notation introduced in~\eqref{eq:notSz},
for sets $X$, $F\subseteq \NN^\l$ and for~$\eps>0$
we write $X\rightarrow_{\eps} F$, if every subset $Y\subseteq X$ with $|Y|\geq \eps |X|$ 
contains a homothetic copy of $F$. 

A simple heuristic, similar to the one in the context of Szem\'eredi's
theorem,  suggests that $n^{-1/(|F|-1)}$ is a lower bound on 
the threshold for the Furstenberg-Katznelson theorem for a configuration 
$F\subseteq \NN^\l$ in the binomial random subset $[n]^\l_{p}$ where elements of $[n]^\l$
are included with probability $p$. 
Our next result shows that, in fact, this gives the correct asymptotic 
for the threshold.
\begin{theorem}
\label{thm:FK}
For every integer $\l\geq 1$, every finite set $F\subseteq \NN^\l$ with $|F|\geq 3$,
 and every constant $\eps\in(0,1)$ there exist $C>c>0$ such that
for any sequence of probabilities $\bq=(q_n)_{n\in\NN}$ we have 
$$
\lim_{n\to\infty}\PP{[n]^\l_{q_n}\rightarrow_\eps F}=
\begin{cases} 
0, & \text{if}\ q_n\leq cn^{-1/(|F|-1)}\ \text{for all}\ n\in\NN,\\
1, & \text{if}\ q_n\geq Cn^{-1/(|F|-1)}\ \text{for all}\ n\in\NN.
\end{cases}
$$
\end{theorem}

\subsection{Density regular matrices}
Another extension of Szemer\'edi's theorem leads to the notion of \emph{density regular} matrices. Arithmetic progressions of length~$k$ can be viewed as the set of distinct-valued
solutions of the following homogeneous system of $k-2$ linear equations
\[
\begin{array}{ccccccc}
x_1 &- &2x_2 &+ &x_3 &= &0\,,\\
x_2 & -&2x_3 &+ &x_4 &= &0\,,\\
\vdots & &\vdots&&\vdots&&\vdots\,\,\\
x_{k-2}& -&2x_{k-1}& + &x_k &=&0\,.
\end{array}
\]
More generally, for an $\l\times k$ integer matrix $A$ let $\cS(A)\subseteq \RR^k$ be 
the set of solutions of the homogeneous system of linear equations given by~$A$. Let 
$\cS_0(A)\subseteq \cS(A)$ be those solutions $(x_1,\dots,x_k)$ with all 
$x_i$ being distinct. We say $A$ is \emph{irredundant} if $\cS_0(A)\neq\emptyset$.
Moreover, an irredundant $\l\times k$ integer matrix~$A$ is \emph{density regular}, 
if for every $\eps>0$ there exists an~$n_0$ such that for all
$n\geq n_0$ and every $Y\subseteq [n]$ with $|Y|\geq \eps n$ we have $Y^k\cap \cS_0(A)\neq \emptyset$.
Szemer\'edi's theorem, for example, implies that the following $(k-2)\times k$ matrix 
\begin{equation}\label{eq:SzA}
\begin{pmatrix}
1 & -2 & 1 & 0 & 0 & \cdots & 0 & 0 &0 \\
0 & 1 & -2 & 1 & 0 & \cdots & 0 & 0 &0 \\
&&&&&\ddots\\
0  & 0 & 0 & 0 & 0 & \cdots & 1 &-2 &1
\end{pmatrix}
\end{equation}
is density regular for any $k\geq 3$.

Density regular matrices are a subclass of so-called partition regular matrices. This class 
was  studied and characterized by Rado~\cite{Ra33} and,  for example, it follows from this characterization that $k\geq \l+2$ (see~\cite{GRS90} for details). In~\cite{FGR88} Frankl, Graham, and R\"odl characterized irredundant, density regular matrices, 
being those partition regular matrices~$A$ for which $(1,1,\dotsc,1)\in\cS(A)$. 

Similar as in the context of Theorem~\ref{thm:Sz} and Theorem~\ref{thm:FK} the following notation will be useful.
For an irredundant, density regular, $\l\times k$ integer matrix $A$, $\eps>0$, and $X\subseteq [n]$ we write
$X\rightarrow_\eps A$ if for every $Y\subseteq X$ with $|Y|\geq \eps |X|$ we have $Y^k\cap \cS_0(A)\neq \emptyset$. The following parameter in connection with Ramsey properties of random subsets of the integers 
with respect to irredundant, partition regular matrices was introduced by R\"odl and Ruci\'nski~\cite{RR97}.

Let $A$ be an $\l\times k$ integer matrix 
and let the columns be indexed by $[k]$.
For  a partition $W\dcup \overline{W}\subseteq [k]$ of the columns of $A$, we denote by 
$A_{\overline{W}}$ the matrix obtained from $A$ by restricting to the columns
indexed by $\overline{W}$. Let $\rank(A_{\overline{W}})$ be the rank of~$A_{\overline{W}}$, where  $\rank(A_{\overline{W}})=0$  for ${\overline{W}}=\emptyset$. We set
\begin{equation}\label{eq:mA}
m(A)=\max_{\substack{W\dcup \overline{W}=[k]\\ |W|\geq 2}} \frac{|W|-1}{|W|-1+\rank(A_{\overline{W}})-\rank(A)}\,.
\end{equation}
It was shown in~\cite{RR97}*{Proposition~2.2~\iti{ii}} that 
for irredundant, partition regular matrices~$A$ the denominator of~\eqref{eq:mA} is always at least $1$.
For example, for $A$ given in~\eqref{eq:SzA} we have $m(A)=k-1$.

It follows from the 0-statement of Theorem~1.1 in~\cite{RR97} that for any irredundant, density regular, $\l\times k$ integer matrix $A$ of rank $\l$ and every $1/2>\eps>0$ there exist a $c>0$ such that 
for every sequence of probabilities $\bq=(q_n)$ with $q_n\leq cn^{-1/m(A)}$ we have
\begin{equation}\label{eq:drlb}
\lim_{n\to\infty} \PP{[n]_{q_n}\rightarrow_\eps A}=0\,.
\end{equation}
We shall deduce a corresponding upper bound from Theorem~\ref{thm:main} and obtain the following result. 
\bigskip

\begin{theorem}
\label{thm:dr}
For every irredundant, density regular, $\l\times k$ integer matrix $A$ with rank~$\l$,
and every $\eps\in(0,1/2)$ there exist constants $C>c>0$ such that
for any sequence of probabilities $\bq=(q_n)_{n\in\NN}$ we have 
$$
\lim_{n\to\infty}\PP{[n]_{q_n}\rightarrow_\eps A}=
\begin{cases} 
0, & \text{if}\ q_n\leq cn^{-1/m(A)}\ \text{for all}\ n\in\NN,\\
1, & \text{if}\ q_n\geq Cn^{-1/m(A)}\ \text{for all}\ n\in\NN.
\end{cases}
$$
\end{theorem}
Note that we restrict $\eps<1/2$ here. With this restriction 
the 0-statement will follow from a result of R\"odl and Ruci\'nski from~\cite{RR97}.
The proof of the 1-statement presented in Section~\ref{sec:pf_dr} actually works for 
all $\eps\in(0,1)$.

\subsection{An extremal problem related to Schur's equation}
In 1916 Schur~\cite{Schur16} 
showed that every partition of the positive integers into finitely many
classes contains a class which contains a solution of the single, homogeneous equation 
$x_1+x_2-x_3=0$. Clearly, the corresponding matrix $\begin{pmatrix}1 &  1& -1\end{pmatrix}$
is not density regular, since the set of all odd integers contains no solution. However, it 
is not hard to show that every subset $Y\subseteq [n]$ with $|Y|\geq (1/2+o(1)) n$
contains such a solution. Similarly, as above for $\eps>0$ and $X\subseteq [n]$ we write 
$$
X\rightarrow_{1/2+\eps} \begin{pmatrix}1 &  1& -1\end{pmatrix}
$$
if every subset $Y\subseteq X$ with $|Y|\geq (1/2+\eps)|X|$ contains 
a distinct-valued solution, i.e., 
\[
	Y^3\cap \cS_0\left( \begin{pmatrix}1 &  1& -1\end{pmatrix} \right)\neq\emptyset\,.
\]

We are interested in the threshold for the extremal problem of Schur's equation, i.e., for the property $X\rightarrow_{1/2+\eps} \begin{pmatrix}1 &  1& -1\end{pmatrix}$. In this context the 
simple heuristic based on the expected number of solutions of the Schur equation 
in random subsets of the integers suggests that $n^{-1/2}$ is the threshold for this property.
Moreover, for 
Schur's theorem in random subsets of the integers the threshold turned out to be 
$n^{-1/2}$ as shown in~\cites{GRR96,FRS}. We show that the threshold of the extremal version of Schur's 
equation is the same.
\begin{theorem}\label{thm:Schur}
For every $\eps\in(0,1/2)$ there exist constants $C>c>0$ such that
for any sequence of probabilities $\bq=(q_n)_{n\in\NN}$ we have 
$$
\lim_{n\to\infty}\PP{[n]_{q_n}\rightarrow_{1/2+\eps}  \begin{pmatrix}1 &  1& -1\end{pmatrix}}=
\begin{cases} 
0, & \text{if}\ q_n\leq cn^{-1/2}\ \text{for all}\ n\in\NN,\\
1, & \text{if}\ q_n\geq Cn^{-1/2}\ \text{for all}\ n\in\NN.
\end{cases}
$$
\end{theorem}

\subsection{Extremal problems for hypergraphs}
\label{sec:tur}
The last result we present here deals with extremal problems for hypergraphs.
An $\l$-uniform hypergraph $H$ is a pair $(V,E)$, where 
the vertex set $V$ is some finite set and the edge set 
$E\subseteq \binom{V}{\l}$ is a subfamily of the $\l$-element 
subsets of $V$. As usual we call 
$2$-uniform hypergraphs simply graphs.
For some hypergraph~$H$ we denote by $V(H)$ and $E(H)$
its vertex set and its edge set and we denote by $v(H)$ and $e(H)$ the 
cardinalities of those sets. For an integer $n$ we denote by 
$K_n^{(\l)}$ the complete $\l$-uniform hypergraph on $n$ vertices, i.e., 
$v(K_n^{(\l)})=n$ and $e(K_n^{(\l)})=\binom{n}{\l}$.  
An $\l$-uniform hypergraph $H'$ is a sub-hypergraph of $H$, if $V(H')\subseteq V(H)$ and 
$E(H')\subseteq E(H)$ and we write $H'\subseteq H$ to denote that.
For a subset $U\subseteq V(H)$ we denote by 
$E(U)$ the edges of~$H$ contained in $U$ and we set $e(U)=|E(U)|$. Moreover, we 
write $H[U]$ for the sub-hypergraph induced  on $U$, i.e., $H[U]=(U,E(U))$. 

For two $\l$-uniform hypergraphs~$F$ and $H$ we say $H$ contains a copy of $F$, if there exists 
an injective map $\phi\colon V(F)\to V(H)$ such that $\phi(e)\in E(H)$ for every $e\in E(F)$. 
If $H$ contains no copy of $F$, then we say $H$ is \emph{$F$-free}.
We denote by $\ex(H,F)$ the maximum number of edges of 
an $F$-free sub-hypergraph of $H$, i.e.,
$$
\ex(H,F)=\max\{e(H')\colon H'\subseteq H \tand H' \ \text{is $F$-free}\}\,.
$$

Mantel~\cite{Ma07}, Erd\H os~\cite{Er38}, and Tur\'an~\cite{Tu41} were the first to study this function
for graphs. In particular, Tur\'an determined $\ex(K_n,K_k)$ for all
integers~$n$ and~$k$. This line of research was continued by  Erd\H os and Stone~\cite{ErSt46} and 
Erd\H os and Simonovits~\cite{ErSi66} and those authors
showed that for every graph $F$ with chromatic number $\chi(F)\geq 3$ we have
\begin{equation}\label{eq:ES}
\ex(K_n,F)=\left(1-\frac{1}{\chi(F)-1)}+o(1)\right)\binom{n}{2}\,,
\end{equation}
where $\chi(F)$ is minimum number $r$ such that there exists a partition $V_1\dcup\dots\dcup V_r=V(F)$
such that $E(V_i)=\emptyset$ for every $i\in[r]$. 
Moreover, it follows from the result of K\"ovari, S\'os, and Tur\'an~\cite{KST54}  (see also~\cite{ErSt46})
that 
\begin{equation}\label{eq:KST}
\ex(K_n,F)=o(n^2)
\end{equation}
for graphs~$F$ with $\chi(F)\leq 2$.

For an $\l$-uniform hypergraph $F$ we define the \emph{Tur\'an density}
$$
\pi(F)=\lim_{n\to\infty} \frac{\ex(K^{(\l)}_n,F)}{\binom{n}{\l}}\,.
$$
For a graph $F$ the Tur\'an density $\pi(F)$ is determined due to~\eqref{eq:ES} 
and~\eqref{eq:KST}. For hypergraphs~\eqref{eq:KST} was extended by Erd\H os~\cite{Er64} to $\l$-partite, $\l$-uniform hypergraphs. Here an $\l$-uniform hypergraph $F$ is $\l$-partite if its vertex set can be partitioned into $\l$ classes, such that every edge intersects every partition class in precisely one vertex.
Erd\H os showed that  
$
\pi(F)=0
$ 
for every $\l$-partite, $\l$-uniform hypergraph~$F$. For other $\l$-uniform hypergraphs 
only a few results are known and, for example, determining $\pi(K_4^{(3)})$ is one of the best known 
open problems in the area. However, one can show that $\pi(F)$ indeed exists for every 
hypergraph~$F$ (see, e.g.~\cite{KNS64}).

We study the random variable $\ex(G^{(\l)}(n,q),F)$ for fixed $\l$-uniform hypergraphs~$F$, where~$G^{(\l)}(n,q)$ denotes the binomial random $\l$-uniform sub-hypergraph of $K_n^{(\l)}$ with edges of 
$K_n^{(\l)}$ included independently with probability~$q$. It is easy to show that 
$$
\ex(H,F)\geq \pi(F)e(H)
$$
for all $\l$-uniform hypergraphs~$H$ and~$F$ (see, e.g.~\cite{JLR00}*{Proposition~8.4} for a proof for graphs).
We are interested in the threshold for the property that a.a.s.\ 
\begin{equation}\label{eq:rTur}
\ex(G^{(\l)}(n,q),F)\leq (\pi(F)+o(1)) e(G^{(\l)}(n,q))\,.
\end{equation}

Results of that sort appeared in the work of Babai, Simonovits, and Spencer~\cite{BSS90} who showed 
that~\eqref{eq:rTur} holds random graphs
when $F$ is a clique and $q=1/2$. Moreover, it follows from an earlier result of Frankl and R\"odl~\cite{FR86}
that the same holds for $F=K_3$
as long as $q\gg n^{-1/2}$. The systematic study for graphs was initiated by Kohayakawa 
and his coauthors. In particular, Kohayakawa, \L uczak, and R\"odl formulated  a conjecture for the threshold of 
Tur\'an properties for random graphs (see Conjecture~\ref{conj:KLR} below).

For an $\l$-uniform hypergraph~$F$ with $e(F)\geq 1$ we set
\begin{equation}\label{eq:mF}
m(F)=\max_{\substack{F'\subseteq F\\ e(F')\geq 1}}d(F')
\quad\text{with}\quad 
d(F')=
\begin{cases} 
\frac{e(F')-1}{v(F')-\l}\,, &\text{if}\ v(F')> \l\\
1/\ell\,, &\text{if}\ v(F')=\l\,.
\end{cases}
\end{equation}
It follows from the definition of $m(F)$, that if $q=\Omega(n^{-1/m(F)})$ then a.a.s.\ 
the number of copies of every sub-hypergraph $F'\subseteq F$ in the random hypergraph 
$G^{(\l)}(n,q)$ has at least the same order of magnitude, as the number of edges of $G^{(\l)}(n,q)$. Recall that a 
similar heuristic gave rise 
to the thresholds in the theorem above.

\begin{conjecture}[{\cite{KLR97}*{Conjecture~1~\iti{i}}}]\label{conj:KLR}
For every graph~$F$ with at least one edge and every $\eps>0$ there exists 
$C>0$ such that for every sequence of probabilities $\bq=(q_n)_{n\in\NN}$ with 
$q_n\geq C n^{-1/m(F)}$ we have 
$$
\lim_{n\to\infty}
\PP{\ex(G(n,q_n),F)\leq (\pi(F)+\eps) e(G(n,q_n))}=1\,.
$$
\end{conjecture}

Conjecture~\ref{conj:KLR} was verified for a few special cases. As already mentioned for $F=K_3$ 
the conjecture follow from a result in~\cite{FR86}. For $F$ being a clique with $4$, $5$, or $6$ vertices the conjecture was verified by Kohayakawa, \L uczak, and R\"odl~\cite{KLR97}, Gerke, Schickinger, and Steger~\cite{GSS04} and Gerke~\cite{Ge05}. Moreover, the conjecture is known to be true when 
$F$ is a cycle due to the work of F\"uredi~\cite{Fu94} (for the cycle of length four) and 
Haxell, Kohayakawa, and \L uczak~\cites{HKL95,HKL96} (see also~\cites{KKS98,Kre97}) and the conjecture is known to be true for trees.
The best current bounds on $q$ for which~\eqref{eq:rTur} holds for $F$ being a clique and for 
arbitrary $F$ were obtained by Szab\'o and Vu~\cite{SV03} and Kohayakawa, R\"odl, and Schacht~\cite{KRS04}.

We verify this conjecture for all graphs $F$ and the natural analogue of this conjecture for hypergraphs.
(For $\l$-partite, $\l$-uniform hypergraphs such a conjecture was made in~\cite{RRS07}*{Conjecture~15}.)
\begin{theorem}\label{thm:Tur}
For every $\l$-uniform hypergraph~$F$ with at least one vertex contained in at least two edges and every 
$\eps\in(0,1-\pi(F))$ there exist constants $C>c>0$ such that for any sequence of probabilities $\bq=(q_n)_{n\in\NN}$ we have 
\begin{multline*}
\lim_{n\to\infty}\PP{
\ex\big(G^{(\l)}(n,q_n),F\big) 
\leq 
(\pi(F)+\eps)
e\big(G^{(\l)}(n,q_n)\big)}\\
=
\begin{cases} 
0, & \text{if}\ q_n\leq cn^{-1/m(F)}\ \text{for all}\ n\in\NN,\\
1, & \text{if}\ q_n\geq Cn^{-1/m(F)}\ \text{for all}\ n\in\NN.
\end{cases}
\end{multline*}
\end{theorem}

In Section~\ref{sec:main_results_pf} we will deduce the 1-statements of Theorems~\ref{thm:FK},~\ref{thm:dr},~\ref{thm:Schur}, and~\ref{thm:Tur}
from the main result, Theorem~\ref{thm:main}, 
which we present in the next section. The proofs of the 0-statements 
will be more elementary and will be also given in Section~\ref{sec:main_results_pf}.

\section{Main technical result}
\label{sec:techical}
The main result will be phrased in 
the language of hypergraphs. We will study sequences of hypergraphs
$\bH=(H_n=(V_n,E_n))_{n\in\NN}$. In the context of Theorem~\ref{thm:Sz} 
one may think of $V_n=[n]$ and $E_n$ being the arithmetic progressions of length~$k$.
In the context of Theorems~\ref{thm:FK},~\ref{thm:dr}, and~\ref{thm:Schur}
the corresponding hypergraphs the reader should have in mind are defined 
in a very similar way. For Theorem~\ref{thm:Tur} one should think of $V_n=E(K_n^{(\l)})$ being the edge set of the complete hypergraph~$K_n^{(\l)}$ and
edges of $E_n$ correspond to copies of $F$ in~$K_n^{(\l)}$.

In order to transfer an extremal result from the classical, 
deterministic setting to the probabilistic setting we will require that 
a stronger quantitative version of the extremal result holds (see Definition~\ref{def:adense} below). Roughly speaking, we will require that 
a sufficiently dense sub-structure not only contains one copy of 
the special configuration (not only one arithmetic progression or not only one copy of $F$), but instead the number of those configurations should be of the same order 
as the total number of those configurations in the given underlying ground set.

\begin{definition}\label{def:adense}
	Let $\bH=(H_n)_{n\in\NN}$ be a sequence of $k$-uniform hypergraphs and $\alpha\geq 0$.
	We say $\bH$ is \textbf{$\alpha$-dense} if the following is true.
	
	For every $\eps>0$ there exist $\zeta>0$ and $n_0$ such that for every $n\geq n_0$
	and every $U\subseteq V(H_n)$ with $|U|\geq (\alpha+\eps)|V(H_n)|$ we have
	$$
	|E(H_n[U])|\geq \zeta |E(H_n)|.
	$$
\end{definition}

The second condition in Theorem~\ref{thm:main} imposes a lower bound on 
the smallest probability for which we can transfer the extremal result to the probabilistic setting (see Definition~\ref{def:Kbdd}).
For a $k$-uniform hypergraph $H=(V,E)$, $i\in[k-1]$, $v\in V$, and $U\subseteq V$ we denote by 
$\deg_i(v,U)$ the number of edges of $H$ containing $v$ and having at least $i$ vertices 
in $U\setminus\{v\}$. More precisely, 
\begin{equation}\label{eq:defdegi}
\deg_i(v,U)=\left|\{e\in E\colon |e\cap (U\setminus\{v\})|\geq i
	\tand
v\in e\}\right|\,.
\end{equation}
For $q\in(0,1)$ we let $\mu_i(H,q)$ denote the expected value of the 
sum over all such degrees squared with $U=V_q$ being the binomial random subset of $V$
$$
\mu_i(H,q)=\EE{\sum_{v\in V}\deg^2_i(v,V_q)}\,.
$$

\begin{definition}\label{def:Kbdd}
	Let $K\geq 1$, let $\bH=(H_n)_{n\in\NN}$ be a sequence of $k$-uniform hypergraphs, and
	let $\bp=(p_n)_{n\in\NN}\in(0,1)^\NN$ be a sequence of probabilities.	We say $\bH$ is \textbf{$(K,\bp)$-bounded} if the following is true.
	
	For every $i\in[k-1]$ there exists $n_0$ such that for every 
	$n\geq n_0$ and $q\geq p_n$ we have
	\begin{equation}\label{eq:Kbdd}
	\mu_i(H_n,q)\leq Kq^{2i}\frac{|E(H_n)|^2}{|V(H_n)|}\,.
	\end{equation}
\end{definition}

With those definitions at hand, we can state the main result. 

\begin{theorem}\label{thm:main}
	Let $\bH=(H_n=(V_n,E_n))_{n\in\NN}$ be a sequence of $k$-uniform hypergraphs,
	let $\bp=(p_n)_{n\in\NN}\in(0,1)^\NN$ be a sequence of probabilities
	satisfying $p_n^k|E_n|\to\infty$ as $n\to\infty$, and 
	let $\alpha\geq 0$ and $K\geq 1$.
	If $\bH$ is $\alpha$-dense and $(K,\bp)$-bounded,
	then the following holds.
	
	For every  $\delta>0$ and $(\omega_n)_{n\in\NN}$ with $\omega_n\to\infty$ as $n\to\infty$
	there exists $C\geq 1$
	such that for every $1/\omega_n>q_n\geq Cp_n$ the following holds a.a.s.\ for $V_{n,q_n}$. 
	For every  subset $W\subseteq V_{n,q_n}$ with $|W|\geq (\alpha+\delta)|V_{n,q_n}|$ we have $E(H_n[W])\neq\emptyset$.
\end{theorem}
The proof of Theorem~\ref{thm:main} is based on induction on $k$ and for the induction we will strengthen the statement
(see Lemma~\ref{lem:main} below).

For a $k$-uniform hypergraph $H=(V,E)$ subsets $W\subseteq U\subseteq V$, and  $i\in\{0,1,\dots,k\}$ we consider those edges of $H[U]$ which have at least 
$i$ vertices in $W$ and we denote this family by 
$$
E_U^i(W)=\{e\in E(H[U])\colon |e\cap W|\geq i\}\,.
$$
Note that 
\begin{equation}\label{eq:E0}
	E_U^0(W)=E(H[U])\qand
    E^k_U(W)=E(H[W])
\end{equation} 
for every $W\subseteq U$.

\begin{lemma}\label{lem:main}
	Let $\bH=(H_n=(V_n,E_n))_{n\in\NN}$ be a sequence of $k$-uniform hypergraphs,
	let $\bp=(p_n)_{n\in\NN}\in(0,1)^\NN$ be a sequence of probabilities satisfying $p_n^k|E_n|\to\infty$ as $n\to\infty$, and 
	let $\alpha\geq 0$ and $K\geq 1$.
	If $\bH$ is $\alpha$-dense and $(K,\bp)$-bounded,
	then the following holds.
	
	For every $i\in[k]$, $\delta>0$, and $(\omega_n)_{n\in\NN}$ with $\omega_n\to\infty$ as $n\to\infty$  
	there exist $\xi>0$, $b>0$, $C\geq 1$, and $n_0$
	such that for all $\beta$, $\gamma\in(0,1]$ with $\beta\gamma\geq \alpha+\delta$, 
	every $n\geq n_0$, every $q$ with $1/\omega_n\geq q\geq Cp_n$ the following holds. 
	
	If $U\subseteq V_n$ with 
	$|U|\geq \beta|V_n|$, then the binomial random subset $U_q$ satisfies with probability at least 
	$$
	1-2^{-bq |V_n|}
	$$
	the following property: For every subset $W\subseteq U_q$ with 
	$|W|\geq \gamma |U_q|$ we have 
	$$
	\left|E^i_U(W)\right|\geq \xi q^i |E_n|\,.
	$$
\end{lemma}

Theorem~\ref{thm:main} follows from 
Lemma~\ref{lem:main} applied with $i=k$, $\beta=1$,  
$\gamma=\alpha+\delta$, and $U=V_n$.

\subsection{Probabilistic tools}
We will use Chernoff's inequality in the following form (see, e.g.,~\cite{JLR00}*{Corollary~2.3}).
\begin{theorem}[Chernoff's inequality]
Let $X\subseteq Y$ be finite sets and $p\in(0,1]$. For every $0<\rho\leq 3/2$ we have
\[
\PP{\big||X\cap Y_p|-p|X|\big|\geq \rho p|X|}\leq 2\exp(-\rho^2p|X|/3)\,.
\eqqed
\]
\end{theorem}

We also use an approximate concentration result for $(K,\bp)$-bounded hypergraphs.
The $(K,\bp)$-boundedness only bounds the expected value of the quantity $\sum_v\deg_i^2(v,V_p)$.
In the proof of Lemma~\ref{lem:main} we need an exponential upper tail bound 
and, unfortunately, it is known that such bounds usually not exist. However, it was shown by 
R\"odl and Ruci\'nski in~\cite{RR95} that at the cost of deleting a few elements 
such bound can be obtained. We will again apply this idea in the proof of Lemma~\ref{lem:main}.

\begin{proposition}[{Upper tail~\cite{RR95}*{Lemma~4}}]\label{prop:Kpbdd}
	Let $\bH=(H_n=(V_n,E_n))_{n\in\NN}$ be a sequence of $k$-uniform hypergraphs,
	let $\bp=(p_n)_{n\in\NN}\in(0,1)^\NN$ be a sequence of probabilities, and 
	let $K\geq 1$.
	If $\bH$ is $(K,\bp)$-bounded, then 
	the following holds.
	
	For every $i\in[k-1]$ and every $\eta>0$ there exist $b>0$ and $n_0$ such that for every 
	$n\geq n_0$ and every $q\geq p_n$ the binomial random subset 
	$V_{n,q}$  has the following property
	with probability at least $1-2^{-bq|V_n|+1+\log_2k}$.
	There exists a set $X\subseteq V_{n,q}$ with $|X|\leq \eta q|V_{n}|$ such that
	$$
	\sum_{v\in V_n}\deg^2_i(v,V_{n,q}\setminus X)\leq 4^kk^2Kq^{2i}\frac{|E_n|^2}{|V_n|}\,.
	$$
\end{proposition}
The proof follows the lines of~\cite{RR95}*{Lemma~4} and we include it for completeness.
\begin{proof}
	Suppose $\bH$ is $(K,\bp)$-bounded and $i\in[k-1]$ and $\eta>0$ are given.
	We set 
	$$
	b=\frac{\eta}{4(k-1)^2}
	$$
	and $n_0$ be sufficiently large, so that~\eqref{eq:Kbdd} holds for every 
	$n\geq n_0$ and $q\geq p_n$.
	
	For every $j=i,\dots,2(k-1)$ we consider the family  $\sS_j$ defined as follows
	\begin{multline*}
	\sS_j=\Big\{(S,v,e,e')\colon S \subseteq V_n,\, v\in V_n,\, e,e'\in E_n\ \text{such that}\ |S|=j,\\
		v\in e\cap e',\ S\subseteq (e\cup e')\setminus \{v\},\ |e\cap S|\geq i\tand |e'\cap S|\geq i\Big\}\,.
	\end{multline*}
					Let $\cS_{j}$ be the random variable denoting the number of elements $(S,v,e,e')$ from $\sS_j$ with  $S\in \binom{V_{n,q}}{j}$.
	By definition we have 
	$\sum_{j=i}^{2k-2}\EE{\cS_j}\leq 4^{k-1}\mu_i(H_n,q)$ and due to the $(K,\bp)$-boundedness of $\bH$ we have
	$$
	\max_{j=i,\dots,2(k-1)}\EE{\cS_j}\leq \sum_{j=i}^{2k-2}\EE{\cS_j}\leq4^{k-1}\mu_i(H_n,q)\leq 4^{k-1}Kq^{2i}\frac{|E_n|^2}{|V_n|}\,.
	$$
	Let $Z_j$ be the random variable denoting the number  of sequences 
	$$
	((S_{r},v_r,e_r,e'_r))_{r\in [z]}\in \sS_j^z
	$$ 
	of length 
	$$
	z=\left\lceil \frac{\eta q |V_n|}{4(k-1)^2}\right\rceil
	\leq \left \lceil \frac{\eta q |V_n|}{2(k-1)j}\right\rceil
	$$
	which satisfy
	\begin{itemize}
		\item[\iti{i}] the sets  $S_r$ are contained in $V_{n,q}$ and
		\item[\iti{ii}] the sets $S_r$ are mutually disjoint, i.e., $S_{r_1}\cap S_{r_2}=\emptyset$ for all $1\leq r_1<r_2\leq z$.
	\end{itemize}
	Clearly,  we have 
	$$
	\EE{Z_j}\leq |\sS_j|^zq^{jz}=  \left(\EE{\cS_j}\right)^z\leq \left(4^{k-1}Kq^{2i}\frac{|E_n|^2}{|V_n|}\right)^z\,.
	$$
	On the other hand,  if 
	$$
	\sum_{v\in V_n}\deg_i^2(v,V_{n,q}\setminus X)\geq 4^kk^2Kq^{2i}\frac{|E_n|^2}{|V_n|}\geq \sum_{j=i}^{2k-2}j\cdot 2\cdot 4^{k-1}Kq^{2i}\frac{|E_n|^2}{|V_n|}
	$$
	for any $X\subseteq V_{n,q}$ with 
	$|X|\leq \eta q |V_n|$, then  
	there exists some $j_0\in\{i,\dots, 2k-2\}$ such that
	$$
	Z_{j_0}\geq \left(2\cdot 4^{k-1}Kq^{2i}\frac{|E_n|^2}{|V_n|}\right)^z\,.
	$$
	Markov's inequality bounds the probability of this event by
	\begin{multline*}
	\PP{\exists j_0\in\{i,\dots,2k-2\}\colon Z_{j_0}\geq 2^z\left(4^{k-1}Kq^{2i}\frac{|E_n|^2}{|V_n|}\right)^z}\\
	\leq
	\sum_{j=i}^{2k-2}\PP{Z_{j}\geq 2^z\left(4^{k-1}Kq^{2i}\frac{|E_n|^2}{|V_n|}\right)^z}
	\leq 2k\cdot 2^{-z}\leq  2^{-b q |V_n|+1+\log_2k}\,,
	\end{multline*}
	which concludes the proof of Proposition~\ref{prop:Kpbdd}.
\end{proof}

\subsection{Proof of Lemma~\ref{lem:main}}
\label{sec:lem_main}
Let $\bH=(H_n=(V_n,E_n))_{n\in\NN}$ be a sequence of $k$-uniform hypergraphs,
let $\bp=(p_n)_{n\in\NN}\in(0,1)^\NN$ be a sequence of probabilities, and 
let $\alpha\geq 0$ and $K\geq 1$ such that $\bH$ is $\alpha$-dense and $(K,\bp)$-bounded.
We prove Lemma~\ref{lem:main} by induction on~$i$.

\paragraph{Induction start ($i=1$)} For  $\delta>0$ and $(\omega_n)_{n\in\NN}$ 
(which plays no role for the induction start) 
we appeal to the $\alpha$-denseness of $\bH$
and let~$\zeta$ and $n_1$ be the constants given by this property for $\eps=\delta/8$. We set
$$
\xi=\frac{\delta\zeta}{8k}\,,\quad 
b=\frac{\delta^3}{193}\,,\quad 
C=1\,,\qand 
n_0=n_1\,.
$$
Let $\beta$, $\gamma\in(0,1]$ satisfy $\beta\gamma\geq \alpha+\delta$, let $n\geq n_0$
be sufficiently large, $q\geq p_n$,
and let $U\subseteq V_n$ with $|U|\geq\beta |V_n|$ be given.
We consider the set $Y\subseteq U$ defined by
$$
Y=\left\{u\in U\colon \left|\{e\in E(H_n[U])\colon u\in e\}\right| \leq \frac{\zeta|E_n|}{2|V_n|} \right\}\,.
$$
In other words, $Y$ is the set of vertices in $U$ with low degree in $H_n[U]$. Due to the $\alpha$-denseness of $\bH$
we have
$$
|Y|\leq \left(\alpha+\frac{\delta}{8}\right)|V_n|\,.
$$
It follows from Chernoff's inequality that with probability at least 
$$
1-2\exp(-\delta^2q|U|/48)-2\exp(-\delta^2q|V_n|/192)\geq 1-2^{-bq|V_n|}
$$
we have 
$$
|U_q|\geq \left(1-\frac{\delta}{4}\right)q|U|
\qand
|U_q\cap Y|\leq \left(\alpha+\frac{\delta}{4}\right)q|V_n|\,.
$$
Consequently, for every $W\subseteq U_q$ satisfying $|W|\geq \gamma |U_q|$
we have
\begin{align*}
|W|
&\geq \gamma |U_q|
\geq \left(1-\frac{\delta}{4}\right)\gamma q|U|
\geq \left(1-\frac{\delta}{4}\right)\beta\gamma q|V_n|\\
&\geq \left(1-\frac{\delta}{4}\right)\left(\alpha+\delta\right)q|V_n|
\geq \left(\alpha+\frac{\delta}{2}\right)q|V_n|
\geq |U_q\cap Y|+\frac{\delta}{4}q|V_n|
\end{align*}
and the definition of $Y$ yields
$$
\big|E_U^1(W)\big|
\geq 
|W\setminus Y|\cdot\frac{1}{k}\frac{\zeta|E_n|}{2|V_n|}
\geq 
\frac{\delta}{4}q|V_n|\cdot\frac{1}{k}\frac{\zeta|E_n|}{2|V_n|}
=
\xi q|E_n|\,.
$$
This concludes the proof of the induction start.

\paragraph{Induction step ($i\longrightarrow i+1$)}
Let $i\geq 1$, $\delta>0$, and $(\omega_n)_{n\in\NN}$ with $\omega_n\to \infty$ as $n\to\infty$ be given. 
We will expose the random set $U_{q}$ in several rounds.
The number of ``main'' rounds~$R$ will depend on the constant $\xi(i,\delta/8)$, which is given by the induction 
assumption. 
More precisely, let
$$
\xi'=\xi(i,\delta/8)\,,\quad
b'=b(i,\delta/8)\,,\quad
C'=C(i,\delta/8)\,,\qand
n'=n_0(i,\delta/8)
$$
be given by the induction assumption applied with $\delta'=\delta/8$.
We set
\begin{equation}\label{eq:R}
R=\left\lceil \frac{4^{k+2}k^2K}{\delta(\xi')^2}+1\right\rceil\,.
\end{equation}

\paragraph{Overview}
Roughly, speaking our argument is as follows.
We will expose $U_q$ in~$R$ main rounds of the same weight, i.e., we will chose 
$q_R$ in such a way that 
$(1-q)=(1-q_R)^R$ and we let $U_q=U^{1}_{q_R}\cup\dots \cup U^{R}_{q_R}$.
Since, every subset $W$, which we have to consider, contains at least $\gamma\geq\alpha+\delta$ proportion of the elements  
of $U_q$ there must be at least $\delta R/4$ rounds such that 
$|U^s_{q_R}\cap W|\geq (\alpha+\delta/2) |U_{q_R}|$. For those rounds we will appeal 
to the induction assumption, which combined with Proposition~\ref{prop:Kpbdd}, implies that~$U$ contains at least 
$\Omega((\xi')^2|V_n|)$ elements $u\in U$ with the property that every such $u$ 
completes ``many'' elements in $E^{i}_U(W\cap U^s_{q_R})$ to elements in $E^{i+1}_U(W\cap U^s_{q_R})$.
 Moreover, 
in each of these ``substantial'' rounds $(\xi')^2|V_n|/(4^{k+1}k^2K)$ new  ``rich'' elements $u$ will be created. 
Consequently, after at most $\delta R/4-1$ of these substantial rounds all but, say,
at most $(\alpha+\delta/8) |V_n|< \gamma|V_n|$ elements of $U$ are rich and in the final substantial round 
$W\cap U_{q_R}$ must contain many rich $u\in U$and therefore create
many elements from $E^{i+1}_U(W)$.

However, the error probabilities in the later rounds will have to beat 
the number of choices for the elements of $W$ in the earlier rounds. For that we will split the 
earlier main rounds into several subrounds. This does not affect 
the argument indicated above, since our bound on 
the number of ``rich'' elements will be independent of~$q_R$. 
We now continue with the details of this proof.

\paragraph{Constants}
Set
\begin{equation}\label{eq:eta}
\eta=\frac{\delta^2}{16}
\end{equation}
and let $\bdel$ and $\ndel$ be given by Proposition~\ref{prop:Kpbdd}
applied with $i$ and $\eta$.
We set 
\begin{equation}\label{eq:b*}
b^*=\min\left\{\frac{\delta^4}{10^6}\,,\frac{b'}{3}\,,\frac{\bdel}{3}\right\}
\qand
B=\left\lceil 1+\frac{1.01^2}{b^*}\right\rceil\,.
\end{equation}
Finally, let 
\begin{align}
\xi&=\frac{\xi'\delta^2}{18k(RB^{R-1})^{i+1}}\,,\label{eq:xi}\\
b&=\min\left\{\frac{\delta^3}{60001RB^{R-1}}\,,\frac{b^*}{2RB^{R-1}}\right\}\,,\label{eq:b}\\
C&=RB^{R-1}C'\,,\label{eq:C}
\end{align}
and let 
$n_0\geq \max\{n',\ndel\}$ be sufficiently large. 

Suppose  $\beta$ and $\gamma\in(0,1]$ satisfy
$$
\beta\gamma\geq \alpha+\delta\,.
$$ 
Let 
$n\geq n_0$ and let $q$ satisfy $1/\omega_n\geq q\geq Cp_n$.
Moreover, let $U\subseteq V_n$ be such that $|U|\geq \beta |V_n|$. Note that 
$$
\min\{\beta,\gamma\}\geq \alpha+\delta\geq \delta>0 \qqand |U|\geq (\alpha+\delta)|V_n|\,.
$$

For a simpler notation from now on we suppress the subscript $n$ in $p_n$,
$H_n$, $V_n$ and $E_n$. 

\paragraph{Details of the induction step}
As discussed above we generate the random set $U_q$ in several rounds.
We will have $R$ main rounds and for that we choose $q_R$ such that
$$
1-q=(1-q_R)^R\,.
$$
For $s\in[R]$ we will further split the $s$th  main round into 
$B^{R-s}$ subrounds. For $s\in[R]$ we set 
$$
r_s=B^{R-s}
$$
and let $q_s$ satisfy
$$
(1-q_R)=(1-q_s)^{r_s}\,.
$$
Note that for sufficiently large $n$, due to $q_n\leq 1/\omega_n$ and $\omega_n\to\infty$ we have
\begin{equation}
\label{eq:lbq}
\left(1+\frac{\delta}{100}\right)\frac{q}{R}\geq q_R\geq \frac{q}{R}
\qand
\left(1+\frac{\delta}{100}\right)\frac{q_R}{B^{R-s}}\geq q_s
\geq \frac{q_R}{B^{R-s}}\,,
\end{equation}
and due to the choice of $B$ we have
\begin{equation}\label{eq:lbq2}
\sum_{t=1}^{s-1}q_t\leq 1.01\frac{q_R}{B^R}\sum_{t=1}^{s-1}B^t
\overset{\eqref{eq:b*}}{\leq}
\frac{b^*}{1.01}\frac{q_R}{B^{R}}B^s\leq \frac{b^*}{1.01}q_s\,.
\end{equation}

We proceed as follows we first 
consider $r_1$ rounds with probability $q_1$, which all together 
establish the first main round and we denote the random subsets obtained by
$$
U_{q_R}^1=U_{q_1}^{1,1}\cup\dots\cup U_{q_1}^{1,r_1}\,.
$$
This is followed by $r_2$ rounds with probability $q_2$ establishing the second main round.
This way we have
$$
U_q=U^1_{q_R}\cup\dots\cup U^R_{q_R}
$$
and for all $s\in[R]$
$$
U^s_{q_R}=U^{s,1}_{q_s}\cup\dots\cup U_{q_s}^{s,r_s}\,.
$$
Furthermore, let $W\subseteq U_q$ with $|W|\geq \gamma|U_q|$ and let 
$$
W^{s}=W\cap U^s_{q_R}
\qand
W^{s,j}=W\cap U^{s,j}_{q_s}
$$
for all $s\in[R]$ and $j\in[r_s]$. 

In our analysis we focus on ``substantial'' rounds.
For that let $S\subseteq [R]$ be the set defined by $s\in S$ if and only if
$$
|W^s|\geq\left(\gamma-\frac{\delta}{2}\right)|U^s_{q_R}|\,.
$$
By definition of $S$, for every $s\in S$ exists some $j_s\in [r_s]$ such that 
$$
|W^{s,j_s}|\geq\left(\gamma-\frac{\delta}{2}\right)|U^{s,j_s}_{q_s}|
$$
and for the rest of the proof we fix such an $j_s$  for every $s\in S$.
The following claim is a direct consequence of Chernoff's inequality.
\begin{claim}\label{claim:1}
Let $\cA$ denote the event that $|S|\geq \delta R/4$. Then 
$\PP{\cA}\geq 1-2^{-2bq|V|}$.
\end{claim}
\begin{proof}
Due to Chernoff's inequality we have 
\begin{equation}\label{eq:ch1}
|U^{s,j}_{q_s}|=(1\pm 0.01\delta) q_{s} |U|\,.
\end{equation}
for all $s\in[R]$ and every $j\in[r_s]$ 
with probability at least
$$
1-2\sum_{s=1}^R r_s\exp(-\delta^2q_s|U|/30000)
\geq
1-2^{-2bq|V|}\,,
$$
where we used 
$q_1\leq q_s$,~\eqref{eq:lbq}, the choice of $b$ in~\eqref{eq:b} 
and the fact that $n$ is sufficiently large for the last inequality.
Since $|W|\geq \gamma |U_q|$ we have
\[
|S|
\geq \frac{|W|-R\cdot (1+\delta/100)(\gamma-\delta/2)q_R|U|}{(1+\delta/100) q_R |U|}
\geq \frac{(1-\delta/100)\gamma q}{(1+\delta/100) q_R}-\left(\gamma-\frac{\delta}{2}\right)R 
\overset{\eqref{eq:lbq}}{\geq} \frac{\delta}{4}R
\]
with probability at least $1-2^{-2bq|V|}$.
\end{proof}

For the rest of the proof we analyze the rounds indexed by $(s,j_s)$ for $s\in S$.
For $s\in S$ we set
$$
W(s)=\bigcup_{\subalign{t&\in S\\t&\leq s}} W^{t,j_t}
\qqand
U(s)=\bigcup_{\subalign{t&\in S\\t&\leq s}} U^{t,j_t}_{q_t}\,.
$$ 
Note that $W(t)=U(t)=\emptyset$ for all $t<\min_{s\in S}s$.
Roughly speaking, we will show for every $s\in S$ that 
either $E^{i+1}_U(W(s))$ is sufficiently large 
or
$\Omega(|V|)$ new ``rich''
elements in~$U$ will be created. 
More precisely, for $s\in S$ we consider the following subset $Z^s\subseteq U$ of rich elements
$$
Z^s:=\left\{u\in U\colon \deg_{i}(u, W^{s,j_s},U)\geq \frac{\xi'}{2}q_s^{i}\frac{|E|}{|V|}\right\}\,,
$$
where
\begin{multline}\label{eq:deg}
\deg_{i}(u, W^{s,j_s},U):=\big|\big\{e\in E\colon |e\cap (W^{s,j_s}\setminus\{u\})|\geq i, u\in e, \tand
e\subseteq U\big\}\big|\,.
\end{multline}
Note that 
$
\deg_{i}(u, W^{s,j_s},V)=\deg_{i}(u, W^{s,j_s})
$
and, hence, for every set $U\subseteq V$ and every $u\in V$ we have
\begin{equation}\label{eq:deg2}
\deg_{i}(u, W^{s,j_s},U)\leq\deg_{i}(u, W^{s,j_s})\,.
\end{equation}
Similarly, as above we set 
$$
Z(s)=\bigcup_{\subalign{t&\in S\\t&\leq s}} Z^s
$$
\begin{claim}\label{claim:main}
For every $s\in S$ and any choice of $W(s-1)\subseteq U(s-1)$ let 
$\cB_{W(s-1)}$ denote the event that 
$U^{s,j_s}_{q_s}$ satisfies the following properties:
\begin{itemize}
\item[\iti{i}] $|U^{s,j_s}_{q_s}|\leq 1.01q_s|U|$ and
\item[\iti{ii}] for every $W^{s,j_s}$ with $|W^{s,j_s}|\geq (\gamma-\delta/2)|U^{s,j_s}_{q_s}|$
either 
\begin{equation}\label{eq:case1}
|E^{i+1}_U(W(s))|
\geq
\xi q^{i+1}|E|
\end{equation}
or  
\begin{equation}\label{eq:case2}
|Z(s)\setminus Z(s-1)|\geq \frac{(\xi')^2}{4^{k+1}k^2 K}|V|\,.
\end{equation}
\end{itemize}
Then 
$$
\PP{\cB_{W(s-1)}\mid U(s-1)}\geq 1-2^{-2b^*q_{s}|V|}\,,
$$
where $\PP{\cB_{W(s_0-1)}\mid U(s_0-1)}=\PP{\cB_{W(s_0-1)}}$ for $s_0=\min_{s\in S}s$. 
\end{claim}
Before we verify Claim~\ref{claim:main} we deduce  Lemma~\ref{lem:main} from it.
Let $\cC$ denote the event that the conclusion of Lemma~\ref{lem:main} holds.
If event $\cA$ holds and $\cB_{W(s-1)}$ holds for every $s\in S$, then $\cC$ must hold, since 
\eqref{eq:case2} in Claim~\ref{claim:main} can occur at most
$$
\frac{4^{k+1}k^2K}{(\xi')^2}\overset{\eqref{eq:R}}{<}\frac{\delta}{4} R\leq |S|
$$
times and, therefore, \eqref{eq:case1} in Claim~\ref{claim:main} must occur.
Below we will verify that  this happens with a sufficiently large probability.
Setting $\PP{U(s_0-1)}=1$ for $s_0=\min_{s\in S}s$, we have
\begin{align*}
\PP{\neg \cC}&
\leq \PP{\neg \cA} 
+
\sum_{S\subseteq [R]} 
\sum_{s\in S}
\sum_{U(s-1)} 
\sum_{W(s-1)}
\PP{\neg \cB_{W(s-1)}\mid U(s-1)}\PP{U(s-1)}\,,
\end{align*}
where the first sum runs over all subsets $S\subseteq [R]$ with $|S|\geq \delta R/4$, 
the third sum runs over all choices of 
$U(s-1)=\bigcup_{t\in S, t<s} U_{q_t}^{t,j_t}$ with 
$|U_{q_t}^{t,j_t}|\leq 1.01q_t|U|$, and the inner sum
runs over all $2^{1.01|V|\sum_{t\in S, t<s}q_t}$ 
choices of $W(s-1)\subseteq U(s-1)$.
Therefore, Claims~\ref{claim:1} and~\ref{claim:main} yield
\begin{align*}
\PP{\neg \cC}
&\leq 
2^{-2bq|V|} + 
2^R\sum_{s=1}^R 2^{1.01|V|\sum_{t=1}^{s-1}q_t}\cdot 2^{-2b^*q_s |V|}\\
&\overset{\makebox[0pt]{\scriptsize\eqref{eq:lbq2}}}{\leq}
2^{-2bq|V|} +2^RR2^{-b^*q_{1}|V|}
\overset{\eqref{eq:lbq}}{\leq}
2^{-2bq|V|} +2^RR2^{-b^*q|V|/(RB^{R-1})}
\overset{\eqref{eq:b}}{\leq}
2^{-bq|V|}\,,
\end{align*}
where the last inequality holds for sufficiently large $n$. 
This concludes the proof of Lemma~\ref{lem:main} and it is left to verify
Claim~\ref{claim:main}.\qed

\begin{proof}[Proof of Claim~\ref{claim:main}]
Let $s\in S$, $W(s-1)\subseteq U(s-1)$ be given. Note that this also defines  $Z(s-1)$.
We first observe that  property~\iti{i} of Claim~\ref{claim:main} holds with high probability.
In fact, due to Chernoff's inequality, 
with probability at least
$$
1-2\exp(-\delta^2q_s|U|/30000)\overset{\eqref{eq:b*}}{\geq} 1-2^{-3b^*q_s|V|}
$$
we even have  
\begin{equation}\label{eq:UCh}
|U_{q_s}^{s,j_s}|=(1\pm0.01\delta)q_s|U|
\end{equation}
and below we assume that~\eqref{eq:UCh} holds.
We distinguish two cases for property~\iti{ii}.
\begin{case}[$|U\setminus Z(s-1)|<(\gamma-3\delta/4) |U|$]\rm
Due to Chernoff's inequality with probability at least
$$
1-2\exp(-\delta^2 (\alpha+\delta/4)q_s |U|/192)
\overset{\eqref{eq:b*}}{\geq}
1-2^{-3b^*q_s|V|}
$$
we have
$$
|U^{s,j_s}_{q_s}\setminus Z(s-1)|\leq \left(\gamma-\tfrac{5}{8}\delta\right)|U^{s,j_s}_{q_s}|\,.
$$
Since $s\in S$  it follows that
$$
|W^{s,j_s}\cap Z(s-1)|\geq \frac{\delta}{8} |U^{s,j_s}_{q_s}|
\overset{\eqref{eq:UCh}}{\geq} 0.99\frac{\delta}{8} q_s |U|
\geq 
\frac{\delta\beta}{9} q_s |V|\geq \frac{\delta^2}{9} q_s |V|\,.
$$
Hence the definition of $Z(s-1)\subseteq  \bigcup_{t\in S, t< s}Z^s$  and $q_1\leq q_t$ for all $t\in S$
yields
\begin{align*}
|E_U^{i+1}(W(s))|&\geq
\frac{\delta^2}{9} q_s |V| \cdot \frac{1}{k}\frac{\xi'}{2}q_1^{i}\frac{|E|}{|V|}\\
&\geq
\frac{\xi'\delta^2}{18k}q_1^{i+1}|E| 
\overset{\eqref{eq:lbq}}{\geq} 
\frac{\xi'\delta^2}{18k(RB^{R-1})^{i+1}}q^{i+1}|E| \overset{\eqref{eq:xi}}{\geq} \xi q^{i+1}|E|\,.
\end{align*}
In other words, for this case we showed that 
alternative~\eqref{eq:case1} happens with probability at least $1-2\cdot 2^{-3b^*q_s|V|}\geq 1-2^{-2b^*q_s|V|}$.
\end{case}
\begin{case}[$|U\setminus Z(s-1)|\geq (\gamma-3\delta/4) |U|$]\rm
In this case we consider 
$$
U'=U\setminus Z(s-1)\,.
$$
We set 
$$
\beta'= \frac{|U'|}{|V|}
\qand
\gamma'=\left(\gamma-\frac{7\delta}{8}\right)\frac{|U|}{|U'|}\,.
$$
Clearly, 
$\beta'\in(0,1]$,
$$
0<\gamma'\leq \frac{\gamma-7\delta/8}{\gamma-3\delta/4}\leq 1\,,
$$ 
and
$$
\beta'\gamma'=\left(\gamma-\frac{7\delta}{8}\right)\frac{|U|}{|V|}
\geq \left(\gamma-\frac{7\delta}{8}\right)\beta
\geq \gamma\beta -\frac{7\delta}{8}\geq \alpha+\frac{\delta}{8}\,.
$$
Hence,  
we can apply the induction assumption to $U'$. More precisely, the induction assumption asserts 
that with probability at least 
$$
1-2^{b'q_s|V|}
$$
every subset $\hW'\subseteq  U'_{q_s}$ with $\hW'\geq \gamma' |U'_{q_s}|$ satisfies
\begin{equation}\label{eq:IA}
\left|E^i_{U'}(\hW')\right|\geq \xi' q_s^{i} |E|\,.
\end{equation}
Note that, in fact,
$$
q_s
\overset{\eqref{eq:lbq}}{\geq} 
\frac{q}{RB^{R-1}}
\geq 
\frac{Cp}{RB^{R-1}}
\overset{\eqref{eq:C}}{\geq} 
C'p\,.
$$

We split the random subset $U^{s,j_s}_{q_s}=U'_{q_s}\dcup U''_{q_s}$, where
$$
U'_{q_s}=U^{s,j_s}_{q_s}\setminus Z(s-1)
\qand
U''_{q_s}=U^{s,j_s}_{q_s}\setminus U'_{q_s}\,.
$$ 
Similarly, we split $W^{s,j_s}=W'\dcup W''$ where 
$W'=W^{s,j_s}\cap U'_{q_s}$ and $W''=W^{s,j_s}\cap U''_{q_s}$.

It follows again from Chernoff's
inequality that 
\begin{equation}\label{eq:U'}
|U'_{q_s}|= \left(1\pm\frac{\delta}{16}\right)q_s|U'|
\end{equation}
holds with probability at least
$$
1-2\exp(-\delta^2q_s|U'|/768)\overset{\eqref{eq:b*}}{\geq} 1-2^{-3b^*q_s|V|}\,.
$$
We distinguish two sub-cases depending on the size of $W''$.
\begin{subcase}[$|W''|> \delta |U^{s,j_s}_{q_s}|/8$]\rm 
In this case, it follows from the $W''\subseteq Z(s-1)$
\begin{align*}
|E_U^{i+1}(W(s))|&\geq|W''|\cdot \frac{1}{k}  \frac{\xi'}{2}q_1^{i}\frac{|E|}{|V|}
\geq \frac{\delta}{8} |U^{s,j_s}_{q_s}|\cdot \frac{\xi'}{2k}q_1^{i}\frac{|E|}{|V|}\\
&\overset{\makebox[0pt]{\scriptsize\eqref{eq:UCh}}}{\geq}
 \frac{\delta}{9}q_s|U|\cdot \frac{\xi'}{2k}q_1^{i}\frac{|E|}{|V|}
 \geq \frac{\delta\beta}{9}q_s\cdot \frac{\xi'}{2k}q_1^{i}|E|
 \geq \frac{\delta^2\xi'}{18k}q_1^{i+1}|E|
 \overset{\eqref{eq:xi}}{\geq} \xi q^{i+1}|E|\,.
\end{align*}
In other words, for this case we showed that 
alternative~\eqref{eq:case1} happens with probability at least $1-2\cdot 2^{-3b^*q_s|V|}\geq 1- 2^{-2b^*q_s|V|}$.
\end{subcase}
\begin{subcase}[$|W''|\leq \delta |U^{s,j_s}_{q_s}|/8$]\rm
In this case we appeal to the $(K,\bp)$-boundedness of $\bH$. It follows from Proposition~\ref{prop:Kpbdd}
and the choice of $\eta$ in~\eqref{eq:eta} that with probability at least 
$$
1-2^{-\bdel q_s |V|+1+\log_2k}
$$
there exists a set $X\subseteq U'_{q_s}$ such that
\begin{equation}\label{eq:eta2}
|X|\leq \eta q_s |V|
\overset{\eqref{eq:eta}}{\leq}
\frac{\delta^2}{16}q_s|V|
\leq \frac{\delta}{16}(\alpha+\delta)q_s|V|
\leq \frac{\delta}{16}\beta q_s|V|
\leq \frac{\delta}{16}q_s|U|
\overset{\eqref{eq:UCh}}{\leq} \frac{\delta}{8}|U_{q_s}^{s,j_s}|
\end{equation}
and
\begin{align}\label{eq:DB}
\sum_{u\in U'}\deg_{i}^2(u,W'\setminus X,U')
&\overset{\makebox[0pt]{\scriptsize\eqref{eq:deg2}}}{\leq}
\sum_{u\in U'}\deg_{i}^2(u,W'\setminus X)\nonumber\\
&\leq
\sum_{u\in U'}\deg_{i}^2(u,U'_{q_s}\setminus X)
\leq
4^kk^2Kq_s^{2i}\frac{|E|^2}{|V|}
 \,.
\end{align}
Consider the set 
$$
\hW'=W'\setminus X\,.
$$
Since $s\in S$, it follows from~\eqref{eq:eta2} and the assumption of this case that 
$$
|\hW'|
\geq |W^{s,j_s}|-|W''|-|X|
\geq \left(\gamma-\frac{\delta}{2}\right)|U^{s,j_s}_{q_s}|
- 2\frac{\delta}{8}|U^{s,j_s}_{q_s}|
\geq \left(\gamma-\frac{3\delta}{4}\right)|U^{s,j_s}_{q_s}|\,.
$$
Furthermore assertions~\eqref{eq:UCh} and~\eqref{eq:U'} yield
$$
\frac{|\hW'|}{|U'_{q_s}|}
\geq 
\left(\gamma-\frac{3\delta}{4}\right)\frac{|U^{s,j_s}_{q_s}|}{|U'_{q_s}|}
\geq \frac{(\gamma-3\delta/4)(1-\delta/100)}{1+\delta/16}\frac{|U|}{|U'|}
\geq \left(\gamma-\frac{7\delta}{8}\right)\frac{|U|}{|U'|}=\gamma'
$$
In other words, $\hW'$  satisfies $|\hW'|\geq \gamma'|U'_{q_s}|$ and from the induction assumption
we infer that~\eqref{eq:IA} holds with probability at least 
$1-2^{-b'q_s|V|}$ and then
\begin{equation}\label{eq:IA2}
 \sum_{u\in U'}\deg_{i}(u, \hW',U')\geq |E^{i}_{U'}(\hW')|\geq \xi'q_s^{i} |E|\,.
\end{equation}
For 
$$
\hZ=\left\{u\in U'\colon \deg_{i}(u, \hW',U')\geq \frac{\xi'}{2}q_s^{i} \frac{|E|}{|V|}\right\}
$$
it follows from the  Cauchy-Schwarz inequality
\begin{align*}
4^kk^2Kq_s^{2i}\frac{|E|^2}{|V|}
\overset{\eqref{eq:DB}}{\geq} \sum_{u\in U'}\deg_{i}^2(u,\hW',U')
&\geq \sum_{u\in \hZ}\deg_{i}^2(u,\hW',U')\\
&\geq \frac{1}{|\hZ|}\left(\sum_{u\in \hZ}\deg_{i}(u,\hW',U')\right)^2
\overset{\eqref{eq:IA2}}{\geq} \frac{1}{|\hZ|}\left(\frac{\xi'q_s^{i} |E|}{2}\right)^2\,.
\end{align*}
Consequently, 
$$
|\hZ|\geq \frac{(\xi')^2}{4^{k+1}k^2K}|V|\,.
$$
Since $\hZ\subseteq U'=U\setminus Z(s-1)$ we have 
$\hZ$ is disjoint from $Z(s-1)$. Furthermore, by definition of $\hZ$ we have 
$\hZ\subseteq Z^s$.
Therefore, 
\eqref{eq:case2} of Claim~\ref{claim:main} holds with probability at least
$$
1-2\cdot2^{-3b^*q_s|V|}-2^{-\bdel q_s|V|+1+\log_2k}-2^{-b'q_s|V|}
\overset{\eqref{eq:b*}}{\geq}
1-2^{-2b^*q_s|V|}
\,,
$$
which concludes the proof of Claim~\ref{claim:main}.\qed
\end{subcase}
\end{case}\let\qed\relax
\end{proof}

\section{Proof of the new results}\label{sec:main_results_pf}
In this section we prove Theorems~\ref{thm:Sz},~\ref{thm:FK},~\ref{thm:dr},~\ref{thm:Schur}, and~\ref{thm:Tur}.
While the involved 0-statements will follow from standard probabilistic arguments, the 1-statement of those 
results will follow from Theorem~\ref{thm:main}.

\subsection{Proof of Theorems~\ref{thm:Sz} and~\ref{thm:FK}}
Clearly Theorem~\ref{thm:Sz} follows from Theorem~\ref{thm:FK}
applied with $\l=1$ and $F=[k]$ and it suffices to verify Theorem~\ref{thm:FK}.

\subsubsection*{The 0-statement of Theorem~\ref{thm:FK}} We start with the 0-statement of 
the theorem. Let $F\subseteq \NN^\l$ be a finite subset with $|F|\geq 3$ and $\eps>0$ be given and set 
$$
c=\left(\frac{1-2\eps}{2}\right)^{1/(|F|-1)}\,.
$$ 
We distinguish different cases depending on the sequence $\bq=(q_n)$.
\setcounter{case}{0}
\begin{case}[$q_n\ll n^{-(\l+1)/|F|}$]\rm In this case the expected number of homothetic 
copies of $F$ in~$[n]^\l_{q_n}$ tends to 0. Hence, we infer from Markov's inequality
that a.a.s.\ $[n]^\l_{q_n}$ contains no homothetic copy of $F$, which yields the claim in that range.
\end{case}
\begin{case}[$n^{-\l}\ll q_n \ll n^{-1/(|F|-1)}$]\rm In this range the expected number of homothetic 
copies of $F$ in $[n]^\l_{q_n}$ is asymptotically smaller than the expected number of elements in $[n]^\l_{q_n}$. Moreover, it follows from Chernoff's inequality that a.a.s.\ $|[n]^\l_{q_n}|$
is very close to its expectation. Consequently, it follows from Markov's inequality that 
a.a.s.\ the number of homothetic copies of $F$ in $[n]^\l_{q_n}$ is $o(|[n]^\l_{q_n}|)$. Therefore,
by removing one element from every homothetic copy of $F$ in  $|[n]^\l_{q_n}|$ a.a.s.\ 
we obtain a subset
 $Y$ of size $|Y|\geq \eps|[n]^\l_{q_n}|$, which contains no homothetic copy of $F$ at all, which yields the 
 0-statement in this  case.
\end{case}
Note that due to $|F|\geq 3$ the ranges considered in Cases 1 and 2 overlap.
Similarly, the range considered in the case below overlaps with the one from Case~2.
\begin{case}[$n^{-(\l+1)/|F|}\ll q_n\leq cn^{-1/(|F|-1)}$]\rm
Again appealing to Chernoff's inequality applied to the size of $[n]^\l_{q_n}$
we infer that it suffices to show that a.a.s.\ 
the number of homothetic copies of $F$ in $[n]^\l_{q_n}$ is at most $(1-2\eps)q_n n^\l$.

Let $Z_F$ be the random variable denoting the number of homothetic copies
of $F$.
Clearly, $\EE{Z_F}\leq q_n^{|F|}n^{\l+1}$ and standard calculations show that 
the variance of $Z_F$ satisfies 
$$
\var{Z_F}
=
O\left(q_n^{2|F|-1} n^{\l+2}+q_n^{|F|}n^{\l+1}\right)\,.
$$
Consequently, Chebyshev's inequality yields
$$
\PP{Z_F\geq 2q_n^{|F|}n^{\l+1}}
\leq 
\frac{\var{Z_F}}{q_n^{2|F|}n^{2\l+2}}
=
O\left(\frac{1}{q_n n^\l}+\frac{1}{q^{|F|}_n n^{\l+1}}\right)
=
o(1)\,,
$$
due to the range of $q_n$ we consider in this case.
Hence, the claim follows from the choice of $c$, which yields
\[
2q_n^{|F|}n^{\l+1}\leq (1-2\eps)q_n n^\l\,.\eqqed
\]
\end{case}
\subsubsection*{The 1-statement of Theorem~\ref{thm:FK}}
We now turn to the 1-statement of Theorem~\ref{thm:FK}.
We first note that if $q_n=\Omega(1)$, then the theorem follows directly 
from Chernoff's inequality combined with the original result of 
Furstenberg and Katznelson. Hence we can assume w.l.o.g.\ $q_n=o(1)$.

Let $F\subseteq \NN^\l$ with $k=|F|\geq 3$ and $\eps\in(0,1)$.
We shall apply Theorem~\ref{thm:main}. For that we consider the following sequence 
of $k$-uniform hypergraphs $\bH=(H_n=(V_n,E_n))_{n\in\NN}$.
Let $V_n=[n]^\l$ and let every homothetic copy of $F$ form an edge in~$E_n$. In particular, $|E_n|=\Theta(n^{\l+1})$.
We set $p_n=n^{-1/(k-1)}$, $\bp=(p_n)_{n\in\NN}$ and $\alpha=0$.
Clearly, for those definitions the conclusion of Theorem~\ref{thm:main} yields 
the 1-statement of Theorem~\ref{thm:FK}.
In order to apply Theorem~\ref{thm:main} we have to verify the following three conditions
\begin{itemize}
\item[\iti{a}] $p_n^k|E_n|\to\infty$ as $n\to\infty$,
\item[\iti{b}] $\bH$ is $\alpha$-dense, and 
\item[\iti{c}] $\bH$ is $(K,\bp)$-bounded for some $K\geq 1$.
\end{itemize}

By definition of $p_n$ and $H_n$ we have
$$
p_n^k|E_n|=\Omega\left(n^{-k/(k-1)}n^{\l+1}\right)=\Omega\left(n^{\l-1/(k-1)}\right)\,,
$$ 
which yields~\iti{a}, as $\l\geq 1$ and $k\geq 3$.

Condition~\iti{b} holds, due to  work of Furstenberg and Katznelson~\cite{FuKa78}.
In fact, it follows from the result in~\cite{FuKa78}, 
that for every configuration $F\subseteq \NN^\l$  and every 
$\eps>0$ there exist $\zeta>0$ and $n_0$ such that for every $n\geq n_0$ every subset 
$U\subseteq [n]^\l$ with $|U|\geq \eps n^\l$ contains at least $\zeta n^{\l+1}$ homothetic copies of $F$.
In other words, $\bH$ is $0$-dense.

Hence, it is only left to verify condition~\iti{c}. We have to show that
for every $i\in[k-1]$ and $q\geq p_n=n^{-1/(k-1)}$ we have
\begin{equation}\label{eq:FKc}
\mu_i(H_n,q)=\EE{\sum_{v\in V_n} \deg_i^2(v,V_q)}
=
O\left(q^{2i}n^{\l+2}\right)
=O\left(q^{2i}\frac{|E_n|^2}{|V_n|}\right)\,.
\end{equation}
It follows from the definition of $\deg_i$ in~\eqref{eq:defdegi} that $\mu_i(H_n,q)$
is the expected number of pairs $(F_1,F_2)$ of homothetic copies of $F$ which share 
at least one point $v$ and at least $i$ points different from $v$ of each copy are contained in $[n]^\l_{q}$.
The expected number of such pairs $(F_1,F_2)$ which share exactly one point can be bounded by
$O\left(q^{2i}n^{\l+2}\right)$.
Since for every fixed homothetic copy $F_1$ there exist only constantly many (independent of $n$) 
other copies $F_2$, which share two points with $F_1$, 
the expected number of such pairs $(F_1,F_2)$ with $|F_1\cap F_2|\geq 2$ is bounded by
$$
O\left(q^in^{\l+1}\right)=O\left(q^{2i}n^{\l+2}\right)\,,
$$
since $q\geq Cp_n\geq Cn^{-1/(k-1)}\geq Cn^{-1/i}$.
Consequently,~\eqref{eq:FKc} holds, which concludes the proof of Theorem~\ref{thm:FK}.
\qed

\subsection{Proof of Theorem~\ref{thm:dr}}
\label{sec:pf_dr} 
The proof of the 0-statement follows directly from the 0-statement of Theorem~1.1 in~\cite{RR97}.
Those authors showed that for every irredundant, density regular $\l\times k$ matrix with rank $\l$
there exists a constant $c>0$ such that for $q_n\leq cn^{-m(A)}$ a.a.s.\
$[n]_{q_n}$ can be partitioned into two classes such that none of them contains a distinct-valued 
solution of the homogeneous system given by~$A$. Clearly, this implies the $0$-statement 
of Theorems~\ref{thm:dr} for every $\eps\in(0,1/2)$.

\subsubsection*{The 1-statement of Theorem~\ref{thm:dr}}
First we note that if $q_n=\Omega(1)$, then the statement follows directly 
from Chernoff's inequality combined with the definition of irredundant, density regular matrix.

Let $A$ be an irredundant, density regular $\l\times k$ integer matrix of rank $\l$ and $\eps>0$
For the application of Theorem~\ref{thm:main} we consider the following sequence 
of $k$-uniform hypergraphs $\bH=(H_n=(V_n,E_n))_{n\in\NN}$.
Let $V_n=[n]$ and for every distinct-valued solution $(x_1,\dots,x_k)$ let $\{x_1,\dots,x_k\}$ 
be an edge of $E_n$.
Moreover we set $p_n=n^{-1/m(A)}$, $\bp=(p_n)_{n\in\NN}$ and $\alpha=0$.
The 1-statement of Theorem~\ref{thm:dr} then follows from the conclusion of Theorem~\ref{thm:main} and
we have to verify the same three conditions~\iti{a}-\iti{c}  as in the proof of the 1-statement of 
Theorem~\ref{thm:FK}.

It was shown in~\cite{RR97}*{Proposition~2.2~\iti{ii}} that $m(A)\geq k-1$ 
and due to Rado's characterization of partition regular matrices (which contains the class of all density regular matrices) we have $k-\l\geq 2$, which yields $|E_n|=\Omega(n^2)$. 
Therefore, we have 
$$
p^k|E_n|=\Omega(n^{-k/(k-1)}\cdot n^2)=\Omega\big(n^{\frac{k-2}{k-1}}\big)
$$
and, hence, condition~\iti{a} is satisfied.

Moreover, based on the Furstenberg-Katznelson theorem from~\cite{FuKa78} it was shown by Frankl, Graham, and R\"odl in~\cite{FGR88}*{Theorem~2}, that the sequence 
of hypergraphs $\bH$ defined above is $0$-dense, i.e., condition~\iti{b} is fulfilled.

Consequently, it suffices to verify that 
$\bH$ is $(K,\bp)$-bounded for some $K\geq 1$. 
For $i\in[k-1]$ and $q\geq p_n=n^{-1/m(A)}$ we have to show that 
$$
\mu_i(H_n,q)=O\left(q^{2i}\frac{|E_n|^2}{n}\right)\,.
$$

Recalling the definitions  of $\mu_i(H_n,q)$ and $H_n=([n],E_n)$
we have
\begin{equation}\label{eq:mu_Rado}
\mu_i(H_n,q)=\EE{\sum_{x\in [n]}\deg_i^2(x,V_{n,q})}=\sum_{x\in [n]}\EE{\deg_i^2(x,V_{n,q})}\,.
\end{equation}
Note that $\EE{\deg_i^2(x,V_{n,q})}$ is the expected number of pairs $(X,Y)\in [n]^k\times [n]^k$ such that
\begin{itemize}
\item[\iti{i}] $x\in X\cap Y$,
\item[\iti{ii}] $X=\{x_1,\dots,x_k\}$ and $Y=\{y_1,\dots,y_k\}$ are solutions of $\cL(A)$, where
$$
A\boldsymbol{x}=A\boldsymbol{y}=\boldsymbol{0}
$$ 
for $\boldsymbol{x}=(x_1,\dots,x_k)^t$ and $\boldsymbol{y}=(y_1,\dots,y_k)^t$, and 
\item[\iti{iii}] $|X\cap ([n]_q\setminus\{x\})|\geq i$ and $|Y\cap ([n]_q\setminus\{x\})|\geq i$.
\end{itemize}
For fixed $x$ and $(X,Y)$ let $w\geq 1$ be the largest integer such that there exist indices
$i_1,\dotsc,i_w$ and  $j_1,\dotsc,j_w$ for which 
\begin{equation}\label{eq:Rad_pf1}
x_{i_1}=y_{j_1},\dotsc,x_{i_w}=y_{j_w}\,.
\end{equation}
Consequently,
\begin{equation}\label{eq:xW1W2}
x\in\{x_{i_1},\dotsc,x_{i_w}\}=\{y_{j_1},\dotsc,y_{j_w}\}
\end{equation}
Set 
$W_1=\{i_1,\dotsc,i_w\}$ and $W_2=\{j_1,\dotsc,j_w\}$.

For fixed sets $W_1$, $W_2\subseteq [k]$ we are going to describe all $(2k-w)$-tuples
$X\cup Y$ satisfying~\iti{ii} and~\eqref{eq:Rad_pf1}. To this end consider the
$2\l\times(2k-w)$ matrix $B$, which arises from two copies $A_1$ and $A_2$ of $A$ with permuted columns.
We set $A_1=(A_{\overline{W}_1}\mid A_{W_1})$
and $A_2=(A_{W_2}\mid A_{\overline{W}_2})$ where for every $\alpha=1,\dotsc,w$ the column of $A_{W_1}$ which is indexed 
by~$i_{\alpha}$ aligns with that column of $A_{W_2}$ which is indexed 
by  $j_\alpha$. Then let
$$
B=\left(\begin{array}{c|c|c}
A_{\overline{W}_1} & A_{W_1} & \boldsymbol{0} \\[1ex]
\hline
\boldsymbol{0} & A_{W_2} & A_{\overline{W}_2}\\[1ex]
\end{array}\right)\,.
$$
Without loss of generality we may assume that $\rank(A_{\overline{W}_1})\geq \rank(A_{\overline{W}_2})$
and, therefore,  
$$
\rank(B)\geq \rank(A)+\rank(A_{\overline{W}_1})\,.
$$
Clearly, the number of $(2k-w)$-tuples
$X\cup Y$ satisfying \iti{ii} and~\eqref{eq:Rad_pf1} equals the number of solutions of the homogeneous 
system given by $B$, which is $O(n^{2k-w-\rank(B)})$. 
Since~$A$ is an irredundant, partition regular matrix, it follows 
from~\cite{RR97}*{Proposition~2.2~\iti{i}} that $\rank(A')=\rank(A)$ for every matrix $A'$ obtained from 
$A$ by removing one column. Consequently, any matrix $B'$ obtained from $B$ by removing one of the middle columns (i.e., one of the $w$ columns 
of $B$ which consist of  a column of $A_{W_1}$ and a column of $A_{W_2}$) satisfies 
$$
\rank(B')\geq \rank(A)+\rank(A_{\overline{W}_1})=\l+\rank(A_{\overline{W}_1})\,.
$$
Therefore, it follows from~\eqref{eq:xW1W2} that the number of such  $(2k-w)$-tuples
that also satisfy condition~\iti{i} for some fixed $x\in [n]$ is at most
\begin{equation}\label{eq:Rad_pf2}
O(n^{2k-w-1-\l-\rank(A_{\overline{W}_1})})\,.
\end{equation}

Finally, we estimate the probability that a $(2k-w)$-tuple
$X\cup Y$ satisfying \iti{i}, \iti{ii}, and~\eqref{eq:Rad_pf1}
also satisfies \iti{iii}. Since $|X\cap Y\cap ([n]_q\setminus\{x\})|=j\leq w-1$ and $q\leq 1$ 
this probability is bounded by
$$
\sum_{j=0}^{w-1}q^{2i-j}=O(q^{2i-w+1})\,.
$$
In view of~\eqref{eq:Rad_pf2} we obtain
\begin{equation}\label{eq:Rad_pf_c}
\sum_{x\in [n]}\EE{\deg_i^2(x,V_{n,q})}
= 
\sum_{x\in [n]}\sum_{w=1}^{k}\sum_{\substack{W_1,W_2\subseteq [k]\\|W_1|=|W_2|=w}}O(n^{2k-w-1-\l-\rank(A_{\overline{W}_1})}q^{2i-w+1})\,.
\end{equation}
Note that if $w=1$, then again due to~\cite{RR97}*{Proposition~2.2~\iti{i}} we have $\rank(A_{\overline{W}_1})=\l$ and, therefore,
the contribution of those terms satisfies
\begin{equation}\label{eq:Rad_pf3}
\sum_{x\in [n]}\sum_{\substack{W_1,W_2\subseteq [k]\\|W_1|=|W_2|=1}}O(n^{2k-2\l-2}q^{2i})
=
O(n^{2k-2\l-1}q^{2i})
=
O\left(q^{2i}\frac{|E_n|^2}{n}\right)\,.
\end{equation}
For $w\geq 2$ and $W_1\subseteq [k]$ with $|W_1|=w$ we obtain from the definition of $m(A)$ and $q\geq n^{-1/m(A)}$ that 
$$
q^{w-1}\geq n^{-w+1-\rank(A_{\overline{W}_1})+\l}\,.
$$
Consequently,
\begin{multline}\label{eq:Rad_pf4}
\sum_{x\in [n]}\sum_{w=2}^{k}\sum_{\substack{W_1,W_2\subseteq [k]\\|W_1|=|W_2|=w}}O(n^{2k-w-1-\l-\rank(A_{\overline{W}_1})}q^{2i-w+1})\\
=
\sum_{x\in [n]}\sum_{w=2}^{k}\sum_{\substack{W_1,W_2\subseteq [k]\\|W_1|=|W_2|=w}}O(n^{2k-2-2\l}q^{2i})\\
=
O(n^{2k-2\l-1}q^{2i})
=
O\left(q^{2i}\frac{|E_n|^2}{n}\right)\,.
\end{multline}

Finally, combining~\eqref{eq:mu_Rado},~\eqref{eq:Rad_pf_c},~\eqref{eq:Rad_pf3}, and~\eqref{eq:Rad_pf4} we obtain
$$
\mu_i(H_n,q)=O\left(q^{2i}\frac{|E_n|^2}{n}\right)\,,
$$
which concludes the proof of the 1-statement of Theorem~\ref{thm:dr}.
\qed

\subsection{Proof of Theorem~\ref{thm:Schur}}
The proof is similar to the proof of Theorem~\ref{thm:FK} and we only sketch the main ideas.
\subsubsection*{The 0-statement of Theorem~\ref{thm:Schur}}
We recall  that for the statement $X\rightarrow_{1/2+\eps} \begin{pmatrix}1 &1 &-1\end{pmatrix}$
we only consider distinct-valued of the Schur equation and we call such a solutions \emph{Schur-triples}.
The expected number of Schur-triples contained in $[n]_{q_n}$ is bounded by $q_n^3n^2$. Consequently,
the 0-statement follows from Markov's inequality if $q_n\ll n^{-2/3}$. In the middle range 
$n^{-1}\ll q_n\ll n^{-1/2}$ it follows, on the one hand, from Chernoff's inequality that a.a.s.\
$|[n]_{q_n}|\geq  q_nn/2$. On the other hand, due to Markov's inequality a.a.s.\ the number
of Schur-triples in $[n]_{q_n}$ is $o(q_n n)$ and, hence, the statement holds in this range of $q_n$.
Finally, if $n^{-2/3}\ll q_n\leq c n^{-1/2}$ for sufficiently small $c>0$, 
then using Chebyshev's inequality one obtains the upper bound of
$$
(1-(1/2+\eps))q_n n/2
$$
on the 
number of Schur-triples in $[n]_{q_n}$, which holds a.a.s. 
Consequently, in view of Chernoff's inequality, a.a.s.\ the random set $[n]_{q_n}$ 
contains a subset of size $(1/2+\eps)|[n]_{q_n}|$, which contains no Schur-triple.\qed

\subsubsection*{The 1-statement of Theorem~\ref{thm:Schur}}
Here the we consider a sequence of $3$-uniform hypergraphs, where $V_n=[n]$ and 
$E_n$ corresponds to all Schur-triples in $[n]$ and we set $p_n=n^{-1/2}$ and $\alpha=1/2$.
For given $\eps\in(0,1/2)$ 
we want to appeal to Theorem~\ref{thm:main} and for that we assume $q_n=o(1)$. 
Again the 1-statement of Theorem~\ref{thm:Schur} follows from Theorem~\ref{thm:main}
and we have to verify the three conditions~\iti{a}-\iti{c}  as in the proof of the 1-statement of 
Theorem~\ref{thm:FK}. 

Condition~\iti{a} follows from the definition of $p_n$ and condition~\iti{c}
follows from similar considerations as in the proof of Theorem~\ref{thm:FK}
for $\l=1$ and $k=3$.

In order to verify condition~\iti{b} we have to show that 
for every $\eps>0$ there exist $\zeta>0$ and $n_0$ such that 
for $n\geq n_0$ every subset $A\subseteq [n]$ with $|A|\geq (1/2+\eps)n$ contains at least~$\zeta n^2$ Schur-triples.

So let $A\subseteq [n]$ satisfy $|A|\geq (1/2+\eps)n$
and set $A_1=A\cap\{1,\dots,(1-\eps)n\}$ (ignoring floors and ceilings).
It follows that for every $z\in A\setminus A_1$ there are 
at least 
$$
\left(\frac{1}{2}+\eps\right)n-\eps n - \frac{(1-\eps)n}{2}
=
\frac{\eps}{2}n
$$ 
pairs 
$x\leq y$ with $x$, $y\in A$ such that $x+y=z$. Hence, if 
$|A\setminus A_1|\geq 3\eps^2n/2$, then~$A$ contains at least 
$3\eps^3n^2/4-n$ Schur-triples and 
the claim follows.  

On the other hand, if $|A\setminus A_1|< 3\eps^2n/2$,
then we have 
$$
|A_1|
\geq
\left(\frac{1}{2}+\eps\right)n-\frac{3\eps^2}{2}n
= 
\left(\frac{1}{2}+\frac{3\eps}{2}\right)(1-\eps) n\,.
$$
In other words, we obtained a density increment of $\eps/2$ on the interval $(1-\eps)n$ and 
the conclusion follows from iterating the above argument.

This concludes the proof of condition~\iti{b} and, therefore, Theorem~\ref{thm:main}
yields the proof of the 1-statement of Theorem~\ref{thm:Schur} for sequences $\bq$
satisfying $q_n=o(1)$. The remaining case, when $q_n=\Omega(1)$ then follows by similar arguments 
as given in~\cite{JLR00}*{Proposition~8.6} and we omit the details.
\qed

\subsection{Proof of Theorem~\ref{thm:Tur}}
\subsubsection*{The 0-statement of Theorem~\ref{thm:Tur}}
Let $F$ be an $\l$-uniform hypergraph with at least one vertex of degree~$2$
and $\eps\in(0,1-\pi(F))$. We set 
$$
c=\frac{1-\pi(F)-\eps}{4}\,.
$$ 

For the proof of the $0$-statement we consider different ranges of $\bq=(q_n)_{n\in\NN}$ 
depending on 
the density of the densest sub-hypergraph of $F$ and depending on $m(F)$.
Let $F'$ be the densest sub-hypergraph of $F$ with $e(F')\geq 1$, i.e., 
$F'$ maximizes $e(F')/v(F')$. Moreover, let $F''$ be one of those sub-hypergraphs 
for which 
$$
d(F'')=m(F)
$$ (see~\eqref{eq:mF} for the definition of those parameters).
Note that $e(F'')\geq 2$, since $F$ contains a vertex of degree at least two.
We consider the following three ranges for~$\bq$.
\setcounter{case}{0}
\begin{case}[$q_n\ll n^{-v(F')/e(F')}$]\rm 
In this range the expected number of 
copies of $F'$ in $G^{(\l)}(n,q_n)$ tends to 0 and, therefore, the statement follows from Markov's inequality.
\end{case}
\begin{case}[$n^{-\l}\ll q_n \ll n^{-1/m(F)}$]\rm 
It follows from the definition of $m(F)$, that in this range the expected number of 
copies of $F''$ in $G^{(\l)}(n,q_n)$ is asymptotically smaller than the expected number of 
of edges of $G^{(\l)}(n,q_n)$. Therefore, applying Markov's inequality 
to the number of copies of $F''$ and Chernoff's inequality to the number of edges 
$G^{(\l)}(n,q_n)$ we obtain that a.a.s.\ the number of copies of $F''$ satisfies 
$o(e(G^{(\l)}(n,q_n)))$. Hence, a.a.s.\ we can obtain an $F''$-free, and consequently, 
an $F''$-free sub-hypergraph of $G^{(\l)}(n,q_n)$ by removing only  $o(e(G^{(\l)}(n,q_n)))$
edges, which yields the statement for this range of $q_n$.
\end{case}
We note that $n^{-\l}\ll n^{-v(F')/e(F')}$ since $F$ contains a vertex of degree~$2$.
In other words, the interval considered in Case~2 overlaps with the interval from 
Case~1.
Similarly, the range considered in the case below overlaps with the one from Case~2.
\begin{case}[$n^{-v(F')/e(F')}\ll q_n\leq cn^{-1/m(F)}$]\rm
Applying again Chernoff's inequality to the random variable $e(G^{(\l)}(n,q_n))$ we see that it suffices to show 
that a.a.s.\ the number of copies of $F''$ 
is at most $(1-(\pi(F)+\eps))q_n n^\l/2$.

Let $Z_{F''}$ be the random variable denoting the number of copies of $F''$
in $G^{(\l)}(n,q_n)$.
Clearly, $\EE{Z_{F''}}\leq q_n^{e(F'')}n^{v(F'')}$ and standard calculations show that 
the variance of $Z_{F''}$ satisfies 
$$
\var{Z_{F''}}
=
O\left(
\frac{q_n^{2e(F'')}n^{2v(F'')}}{\min_{F^*\subseteq F, e(F^*)\geq 1}q_n^{e(F^*)}n^{v(F^*)}}\right)
=
O\left(
\frac{q_n^{2e(F'')}n^{2v(F'')}}{q_n^{e(F')}n^{v(F')}}
\right)\,,
$$
due to the choice of $F'$ being the densest sub-hypergraph of~$F$.
Since $q_n\gg n^{-v(F')/e(F')}$ we have $q_n^{e(F')}n^{v(F')}\to\infty$
and, therefore, 
$$
\var{Z_{F''}}=o\left(q_n^{2e(F'')}n^{2v(F'')}\right)
$$
Consequently, Chebyshev's inequality yields
$$
\PP{Z_{F''}\geq 2q_n^{e(F'')}n^{v(F'')}}
\leq 
\frac{\var{Z_{F''}}}{q_n^{2e(F'')}n^{2v(F'')}}
=
o(1)\,.
$$
Moreover, since $q_n\leq cn^{-1/m(F)}$ and $e(F'')\geq 2$ it follows from the choice of $c$
that
$$
2q_n^{e(F'')}n^{v(F'')}
\leq 
\frac{1-(\pi(F)+\eps)}{2}q_n n^\l\,,
$$
which yields the 0-statement in this case. \qed
\end{case}

\subsubsection*{The 1-statement of Theorem~\ref{thm:Tur}}
Let $F$ be an $\l$-uniform hypergraph with at least one vertex.
For an application of Theorem~\ref{thm:main} we consider the sequence of $k$-uniform hypergraphs 
$\bH=(H_n=(V_n,E_n))_{n\in\NN}$ where $V_n=E(K_n^{(\l)})$ and edges of $E_n$
correspond to copies of~$F$ in~$K_n$. Moreover, we set $p_n=n^{-1/m(F)}$ and $\alpha=\pi(F)$.
Clearly, for this set up the conclusion of Theorem~\ref{thm:main}
yields the 1-statement of Theorem~\ref{thm:Tur} for sequences $\bq$
with $q_n=o(1)$.
In order to apply Theorem~\ref{thm:main} we have to verify the three conditions~\iti{a}-\iti{c}
stated in the proof of the 1-statement of Theorem~\ref{thm:FK}.

Condition~\iti{a} follows from the definitions of $p_n$ and $E_n$ combined. In fact, since 
$F$ contains a vertex of degree at least $2$ we have $m(F)\geq 1/(\l-1)$ and 
$p_n|E_n|=\Omega(n)$.
Such a result was obtained 
by Erd\H os and Simonovits~\cite{ErSi83}*{Theorem~1} and, hence, it is left to verify 
condition~\iti{c} only.

To this end observe that $H_n$ is a regular hypergraph with $\binom{n}{\l}$ vertices and every vertex is contained 
in $\Theta(n^{v(F)-\l})$ edges and that $|E_n|=\Theta(n^{v(F)})$. We will show that 
for $q\geq n^{-1/m(F)}$ and $i\in[k-1]$ we have
\begin{equation*}\label{eq:Ram_pf1}
\mu_i(H_n,q)
=
\EE{\sum_{v\in V_n}\deg_i^2(v,V_{n,q})}
=
\sum_{v\in V}\EE{\deg_i^2(v,V_{n,q})}
=
O\left(q^{2i}\frac{|E_n|^2}{|V_n|}\right)\,.
\end{equation*}
Due to the definition of $\bH$ every $v\in V_n$ 
corresponds to an edge $e(v)$ in $K_n^{(\l)}$. Therefore, the number 
$\EE{\deg_i^2(v,V_{n,q})}$ is the expected number of pairs $(F_1,F_2)$
of copies $F_1 $ and $F_2$ of $F$ in $K_n^{(\l)}$ satisfying
$e(v)\in E(F_1)\cap E(F_2)$
and 
both copies  $F_1 $ and $F_2$  have at least $i$ edges in $E(G^{(\l)}(n,q))\setminus\{e(v)\}$. 
Summing over all such pairs $F_1$ and $F_2$ we obtain
\begin{equation}\label{eq:Ram_pf2}
\begin{split}
\EE{\deg_i^2(v,V_{n,q})}
&\leq
\sum_{F_1, F_2\colon e(v)\in E(F_1)\cap E(F_2)}\sum_{j=0}^{|E(F_1)\cap E(F_2)|-1}q^{2i-j}\\
&=
O\left(\sum_{F_1, F_2\colon e(v)\in E(F_1)\cap E(F_2)}q^{2i-(|E(F_1)\cap E(F_2)|-1)}\right)
\end{split}
\end{equation}
since $q\leq 1$. Furthermore,
\begin{equation}\label{eq:Ram_pf3}
\sum_{F_1, F_2\colon e(v)\in E(F_1)\cap E(F_2)}q^{2i-(|E(F_1)\cap E(F_2)|-1)}
=
O\left(\sum_{J\colon e(v)\in E(J)}n^{2v(F)-2v(J)}q^{2i-(e(J)-1)}\right),
\end{equation}
where the sum on the right-hand side is indexed all hypergraphs $J\subseteq K^{(\l)}_n$ which contain~$e(v)$ 
and which are isomorphic to a sub-hypergraph of~$F$.
It follows from the definition of $m(F)$ and $q\geq n^{-1/m(F)}$ that 
$n^{v(J)}q^{e(J)}=\Omega(qn^\l)$. Combining
this with~\eqref{eq:Ram_pf2} and~\eqref{eq:Ram_pf3} we obtain
\begin{align*}
\EE{\deg_i^2(v,V_{n,q})}
&=
O\left(\sum_{J\colon e(v)\in E(J)}n^{2v(F)-2v(J)}q^{2i-(e(J)-1)}\right)\\
&=
O\left(\sum_{J\colon e(v)\in E(J)}n^{2v(F)-v(J)-\l}q^{2i}\right)\,.
\end{align*}
Moreover, since $v(J)\geq \l$ we have 
$$
\EE{\deg_i^2(v,V_{n,q})}
=
O\left(\sum_{J\colon e(v)\in E(J)}n^{2v(F)-2\l}q^{2i}\right)\,,
$$
and, consequently,
$$
\mu_i(H_n,q)=
\sum_{v\in V_n}O(n^{2v(F)-2\l}q^{2i})\\
=
O(n^{2v(F)-\l}q^{2i})
=
O\left(q^{2i}\frac{|E_n|^2}{|V_n|}\right)\,.
$$
This concludes the proof of condition~\iti{c} and, therefore, Theorem~\ref{thm:main}
yields the proof of the 1-statement of Theorem~\ref{thm:Tur} for sequences $\bq$
satisfying $q_n=o(1)$. The remaining case, when $q_n=\Omega(1)$ then follows by similar arguments 
as given in~\cite{JLR00}*{Proposition~8.6} and we omit the details.
\qed

\subsection*{Acknowledgement} I thank Alan Frieze, Yury Person, and Wojciech Samotij
for their comments on the manuscript. I also thank the referee for her or his detailed work.

\end{document}